\documentclass[bibliography=totoc, 12pt]{amsart}
\usepackage[utf8]{inputenc}
\usepackage[margin=1in]{geometry} 
\usepackage{amsmath,amsthm,amssymb,amsfonts, dsfont, graphicx,mathtools,mathrsfs,tikz,hyperref,physics, stmaryrd, bm, tikz-cd,tkz-euclide,comment, enumerate, appendix, todonotes}
\usepackage{subcaption}
\usepackage[backend=biber,bibencoding=utf8,style=numeric,url=true,
    isbn=false,     
    doi=false,           
    abbreviate=true,            
    giveninits=true, 
    natbib=true,
    hyperref=true]{biblatex}
\usetikzlibrary{decorations.pathreplacing,decorations.markings}

\theoremstyle{plain}
\newtheorem{theorem}{Theorem} 

\newtheorem{lemma}{Lemma}

\numberwithin{equation}{section}

\theoremstyle{definition}

\theoremstyle{remark}

\title{Small solutions to homogeneous and inhomogeneous cubic equations}
\author[C. Bernert]{Christian Bernert}
\address{Mathematisches Institut, Bunsenstraße 3-5, 37073 Göttingen, Germany}
\email{christian.bernert@uni-goettingen.de}

\begin{document}
\maketitle

\begin{abstract}
    We study the solubility of cubic equations over the integers. Assuming a necessary congruence condition, the existence of such solutions is established when the $h$-invariant of $C$ is at least $14$, improving on work of Davenport-Lewis and generalizing the method from Heath-Brown's seminal work in the homogeneous case. We also provide an upper bound on the smallest solution, polynomially in the height of the coefficients. The method also yields new results in the homogeneous case where we generalize and improve on previous work of Browning, Dietmann and Elliott.
\end{abstract}

\section{Introduction}

We study the existence of integer solutions to equations $\phi(\mathbf{x})=0$ where $\phi \in \mathbb{Z}[x_1,\dots,x_n]$ is a (not necessarily homogeneous) integer polynomial of degree $3$. We denote the homogeneous parts of degrees $3,2,1$ and $0$ by $C,Q,L$ and $N$, respectively, so that we can write
\[\phi(\mathbf{x})=C(\mathbf{x})+Q(\mathbf{x})+L(\mathbf{x})+N.\]
Much work has been done in the case of homogeneous equations where Heath-Brown \cite{heath2007cubic} proved that for $n \ge 14$, non-trivial solutions $\mathbf{x} \in \mathbb{Z}^n \backslash \{\mathbf{0}\}$ always exist. Conjecturally, this should be true already in the wider range $n \ge 10$. It is known that no congruence obstructions appear for $n \ge 10$ and that the bound is sharp in this respect.

In the inhomogeneous case, the situation is more complicated. First of all, it is easy to produce examples of cubic polynomials in any number of variables which have a congruence obstruction, e.g.
\[\phi(\mathbf{x})=2C_1(\mathbf{x})+1\]
where $C_1(\mathbf{x}) \in \mathbb{Z}[x_1,\dots,x_n]$ is an arbitrary cubic form.

It is therefore clearly necessary to stipulate the following \textit{Necessary Congruence Condition}:

(NCC) For any prime power $p^k$, the equation $\phi(\mathbf{x}) \equiv 0 \pmod{p^k}$ has a solution.

Slightly more subtly, this necessary condition is still not sufficient in the inhomogeneous case, even in the case of many variables, as the following example of Watson shows: The equation
\[\phi(\mathbf{x})=(2x_1-1)(1+x_1^2+x_2^2+\dots+x_n^2)+x_1x_2\]
is easily seen to satisfy the NCC as soon as $n \ge 5$, but clearly does not admit any integral solution as the absolute value of the first term is always at least $1+x_1^2+x_2^2$ and hence larger than the absolute value of the second term.

It is therefore necessary to make further assumptions on the polynomial $\phi$. Hitherto, this has been realized in two variations:

Browning and Heath-Brown \cite{bhb09}, following work of Heath-Brown \cite{heath_brown_ten} in the homogeneous case, have shown that a solution to $\phi(\mathbf{x})=0$ exists if $\phi$ satisfies the NCC, $n \ge 10$ and the cubic part $C$ is non-singular.

Instead, we will focus on an older approach by Davenport and Lewis \cite{dl64} who introduced the $h$-invariant of a cubic form $C$ to be the least positive integer such that
\[C(\mathbf{x})=\sum_{i=1}^h L_i(\mathbf{x})Q_i(\mathbf{x})\]
for appropriate linear forms $L_1,\dots,L_h$ and quadratic forms $Q_1,\dots,Q_h$. Equivalently, $n-h$ is the largest dimension of a linear subspace contained in the cubic hypersurface defined by the equation $C(\mathbf{x})=0$.

Davenport and Lewis then managed to show, following work of Davenport in the homogeneous case \cite{davenport63}, that a solution to $\phi(\mathbf{x})=0$ exists if $\phi$ satisfies the NCC and $h(C) \ge 17$.

Pleasants \cite{pleasants1975cubic} extended this to cover the range $h(C) \ge 16$.

Given the progress on Davenport's work \cite{davenport63} in the homogenous case made by Heath-Brown \cite{heath2007cubic}, it is natural to ask whether his method can be applied to extend the range to $h(C) \ge 14$.

We achieve this goal, indeed proving an asymptotic formula for the number of solutions in a bounded region. To this end, let us define the counting function
\[N(P)=N(P,\mathcal{B})=\#\{\mathbf{x} \in \mathbb{Z}^n \cap P\mathcal{B}: \phi(\mathbf{x})=0\}.\]
where $\mathcal{B} \subset \mathbb{R}^n$ is a box of the shape
\[\mathcal{B}=\prod_{1 \le i \le n} [z_i-1,z_i+1]\]
for some vector $\mathbf{z} \in \mathbb{R}^n$ with $\vert \mathbf{z}\vert \ge 2$ (so that $\mathcal{B}$ does not contain the origin) and $P\mathcal{B}=\{P\mathbf{x}: \mathbf{x} \in \mathcal{B}\}$. We then prove the following:

\begin{theorem}
\label{thm.h=14asymptotic}
Assume that $\phi=C+Q+L+N$ is of degree $3$, non-degenerate, satisfies the NCC and that $h(C) \ge 14$. If the centre $\mathbf{z}$ of the box $\mathcal{B}$ is a suitable non-singular point of the hypersurface $C(\mathbf{x})=0$, then
\[N(P)=(1+o(1))\mathfrak{S} \cdot \mathfrak{I} \cdot P^{n-3}, \quad P \to \infty\]
where $\mathfrak{S}$ and $\mathfrak{I}$ are the usual singular series and singular integral, respectively. We have $\mathfrak{S}>0$ and $\mathfrak{I}>0$. In particular, there is a solution $\mathbf{x} \in \mathbb{Z}^n$ to $\phi(\mathbf{x})=0$.
\end{theorem}

Note that we do not assume that $C$ is non-degenerate, but that $C$ does not vanish identically. Also note that the restriction $h(C) \ge 14$ automatically implies $n \ge 15$ since Heath-Brown's result is equivalent to saying that a cubic form in $14$ variables has $h(C) \le 13$.

In the culmination of a long series of papers, Watson \cite{watson}  has established the existence of integer solutions to $\phi(\mathbf{x})=0$ under the assumptions $4 \le h(C) \le n-3$ and $n \ge 15$, by an alternative argument. Combining this with Theorem \ref{thm.h=14asymptotic}, we obtain the following:

\begin{theorem}
\label{thm.watsoncor}
Assume that $\phi=C+Q+L+N$ satisfies the NCC and that $h(C) \ge 4$ and $n \ge 16$. Then there is a solution $\mathbf{x} \in \mathbb{Z}^n$ to the equation $\phi(\mathbf{x})=0$.
\end{theorem}

Indeed, under the assumption $n \ge 16$, one of the conditions $h(C) \le n-3$ and $h(C) \ge 14$ is always satisfied, allowing us to deduce Theorem \ref{thm.watsoncor}.

Following a line of investigation revitalized by the work of Browning, Dietmann and Elliott \cite{bde12}, we also provide a bound on the smallest solution $\mathbf{x}$ in Theorem \ref{thm.h=14asymptotic} uniform in the coefficients of the polynomial $\phi$:

\begin{theorem}
\label{thm.smallsolution_h=14}
With the same assumptions as in Theorem \ref{thm.h=14asymptotic}, there exists a vector $\mathbf{x} \in \mathbb{Z}^n$ with $\phi(\mathbf{x})=0$ and
\[\max_i \vert x_i\vert \ll M^{6407n^2},\]
where $M$ is the maximum of the absolute values of the coefficients of $\phi$. The implicit constant is absolute.
\end{theorem}

It is natural to ask whether Watson's method can be made uniform in the coefficients as well, allowing us to obtain a uniform version of Theorem \ref{thm.watsoncor}. However, due to the intricate structure of Watson's argument, it is not immediately transparent to the author whether such an extension is possible.

Our method also delivers new results on the smallest solution in the homogeneous case:
\begin{theorem}
\label{thm.smallsolution_hom}
Let $n \ge 14$ and let $C \in \mathbb{Z}[x_1,\dots,x_n]$ be a cubic form. Then there exists a vector $\mathbf{x} \in \mathbb{Z}^n\backslash \{\mathbf{0}\}$ with $C(\mathbf{x})=0$ and
\[\max_i \vert x_i\vert \ll M^{132484}.\]
If additionally we assume $C$ to be non-singular, then for $n=14$ we can ensure that
\[\max_i \vert x_i\vert \ll M^{2049}.\]
\end{theorem}
This improves on the work of Browning-Dietmann-Elliott \cite{bde12}, who had the same results for $n \ge 17$ and with the exponents $360000$ in the general and $1071$ in the non-singular case. However, as we will explain in more detail later, the result in \cite{bde12}, in particular the lower bound for the singular series, relies on a computational mistake which, if corrected, would yield an exponent larger than our exponent $2049$ and also larger than what our method would yield for $n=17$.

It also improves on work in the author's Master Thesis (unpublished) in which the result for general cubic forms with $n \ge 14$ was established with the slightly larger exponent $141718$. The improvement here comes from a slight variation in the van der Corput differencing argument as well as an improved treatment of the singular integral.

We already record at this point for later convenience that it suffices to prove Theorem \ref{thm.smallsolution_hom} for $n=14$, as we may always set variables to zero. We may also assume that $C$ is non-degenerate since otherwise the result is trivial.

In the result for non-singular forms, we have restricted to the case $n=14$ for convenience. It is clear that the same method would also give similar results for all $n \ge 14$, but here the general case does not immediately reduce to the special case $n=14$, as setting variables to zero could leave us a with a singular form.

We also remark that using the Kloosterman refinement of the circle method, the existence of non-trivial solutions to non-singular cubic forms is known already for $n \ge 10$ due to work of Heath-Brown~\cite{heath_brown_ten} and -- under the assumption of local solubility -- even for $n=9$ due to work of Hooley~\cite{Hooley+1988+32+98}. It would be interesting to work out a bound for the smallest solution using their method and compare it with our result.

\subsection*{Notation}

We use $e(\alpha)=e^{2\pi i\alpha}$ and $e_q(x)=e\left(\frac{x}{q}\right)$ and the notations $\mathcal{O}$ and $\ll$ due to Landau and Vinogradov, respectively. For a subset $\mathcal{A} \subset \mathbb{R}^n$, we use the summation condition $\sum_{\mathbf{x} \in \mathcal{A}}$ as a shorthand for $\sum_{\mathbf{x} \in \mathcal{A} \cap \mathbb{Z}^n}$ i.e. we sum over all the lattice points in the set. For such a vector $\mathbf{x}$ we write $\vert \mathbf{x}\vert:=\max_i \vert x_i\vert$. For a positive number $X>0$, we shall also use the notation $\sum_{\mathbf{x} \le X}$ to mean $\sum_{0 < x_1,\dots,x_n \le X}$ and similarly $\sum_{\mathbf{x} (q)}$ to mean $\sum_{x_1,\dots,x_n \pmod q}$.

\medskip

The letter $\varepsilon$ stands for a sufficiently small positive real number, which by convention may change its value finitely many times. In particular, we may write something like $M^{2\varepsilon} \ll M^{\varepsilon}$. Similarly, the symbol $c$ stands for a sufficiently small constant and we may replace e.g. $c$ by $\frac{c}{2}$ if it is convenient. Implicit constants are usually allowed to depend on $\varepsilon$. If we want to stress that it may depend on a parameter $d$, we write $\mathcal{O}_d$ instead of $\mathcal{O}$. 

\subsection*{Setup}

The general strategy to estimate the counting function $N(P)$ (and hence to prove the existence of integer solutions to $\phi(\mathbf{x})=0$) is to use the Hardy-Littlewood Circle Method. To this end, let
\[S(\alpha)=\sum_{\mathbf{x} \in P\mathcal{B}} e(\alpha \phi(\mathbf{x})).\]
It is then clear by orthogonality that
\[N(P)=\int_0^1 S(\alpha) d\alpha.\]
The next step is to dissect the interval $[0,1]$ of integration into two sets $\mathfrak{M}$ and $\mathfrak{m}$, the major and minor arcs, respectively. In our setup, the major arcs will be defined by taking
\[\mathfrak{M}(a,q)=\left[\frac{a}{q}-\frac{u}{P^3}, \frac{a}{q}+\frac{u}{P^3}\right]\]
for coprime integers $a$ and $q$ with $0 \le a<q \le P_0$ and then defining
\[\mathfrak{M}=\bigcup_{q \le P_0} \bigcup_{(a;q)=1} \mathfrak{M}(a,q)\]
as the major arcs, and the complement $\mathfrak{m}:=[0,1] \backslash \mathfrak{M}$ as the minor arcs. Here, $u$ and $P_0$ are certain parameters which will be chosen as fixed powers of $P$ eventually. In the proofs of Theorem \ref{thm.smallsolution_h=14} and \ref{thm.smallsolution_hom} we will also choose $P$ as a fixed power of $M$ allowing us to deduce $N(P)>0$ and thus establishing the existence of a small solution. We note that these results are certainly true for bounded $M$ (by the above cited results) and so we may assume that $M$ is sufficiently large, if necessary. In the case of Theorem \ref{thm.h=14asymptotic}, we will fix $M$ and let $P \to \infty$.

We note at this point that under the harmless assumption $\frac{2u}{P^3}<\frac{1}{P_0^2}$, the arcs $\mathfrak{M}(a,q)$ will be disjoint. We denote the equivalent assumption $2P_0^2u<P^3$ by $(\mathfrak{M}_1)$ for future reference. Since several more such assumptions will be added in the course of our work, a list of all of them is maintained at the end of the article for the convenience of the reader.

\medskip

We denote the cubic form by $C(\mathbf{x})=\sum_{i,j,k} c_{ijk} x_ix_jx_k$ where we assume the $c_{ijk}$ to be symmetric and integral (as this can be achieved by rescaling the equation with a factor of $6$ if necessary).

For further use let us denote the Hessian matrix of $C$ by $M(\mathbf{x})$, where
\[M(\mathbf{x})_{ij}=\sum_k c_{ijk}x_k,\]
so that the entries of $M(\mathbf{x})\mathbf{y}$ are given by the bilinear forms
\[B_i(\mathbf{x},\mathbf{y})=\sum_{j,k} c_{ijk}x_jy_k.\]
We also denote by $r(\mathbf{x})$ the rank of $M(\mathbf{x})$ and by $r_p(\mathbf{x})$ the $\mathbb{F}_p$-rank of $M(\mathbf{x})$ for a prime $p$.

It will be convenient to assume that the coefficient $c_{111}$ is positive and $\gg M$ and that none of the second partial derivatives of $\phi$ vanish identically. Both of this can always be achieved by a suitable change of coordinates with bounded coefficients.

\section{Davenport's Geometric Condition}

For general cubic forms, an asymptotic formula of the shape $N(P) \asymp P^{n-3}$ cannot always be true, as the example of $C(\mathbf{x})=x_1Q(\mathbf{x})$ for a quadratic form $Q$ with $N(P) \gg P^{n-1}$ clearly shows. On the technical side, to successfully bound the contribution from the minor arcs $\mathfrak{m}$, we need to be able to bound the number of solutions to a certain auxiliary system of bilinear equations.

To this end, slightly varying a definition of \cite{bde12}, let us say that a cubic form $C$ is \textit{$\psi$-good} if the assumption
\begin{equation}
    \label{eq.good_form}
    \#\{\mathbf{x} \in \mathbb{Z}^n: \vert \mathbf{x}\vert<H, r(\mathbf{x})=r\} \ll H^{n-14+r+\varepsilon}
\end{equation}
for all $0 \le r \le n$, holds uniformly in the range $1 \le H \le M^{\psi}$. Note that this estimate is trivially true for $r \ge 14$.

To relate this to the assumptions in our results, we require the following observations: The first is Lemma 28 in \cite{hooley_nonary_2}:

\begin{lemma}
\label{lemma_gc_nonsingular}
If $C$ is non-singular, then
\[\dim \{\mathbf{x}: r(\mathbf{x}) \le r\} \le r.\]
Hence,
\[\#\{\mathbf{x} \in \mathbb{Z}^n: \vert \mathbf{x}\vert<H, r(\mathbf{x})=r\} \ll H^{r+\varepsilon}\]
uniformly over all $H$. In particular, $C$ is $\psi$-good whenever $n \ge 14$.
\end{lemma}

\begin{lemma}
\label{lemma_gc_hinvariante}
If $C$ is not $\psi$-good, then it vanishes on a $(n-13)$-dimensional subspace containing a non-zero element of size $\ll M^{97+91\psi}$. In particular, $h(C) \le 13$.
\begin{proof}
We follow the proof of Lemma 3 in \cite{dl64}. By assumption, for some $H \le M^{\psi}$ and for some $r$, there are more than $H^{n-14+r+\varepsilon}$ points $\mathbf{x}$ with $\vert \mathbf{x}\vert<H$ and $r(\mathbf{x})=r$. Note that this already implies $r \le 13$. This means that for some particular $r \times r$ minor of $M$, the number of points $\mathbf{x}$ with $r(\mathbf{x})=r$ for which this particular minor does not vanish, is more than $H^{r+\varepsilon}$. We can then find $n-r$ linearly independent solutions $\mathbf{y}^{(1)},\dots,\mathbf{y}^{(n-r)}$ for each such $\mathbf{x}$, which indeed depend polynomially on $\mathbf{x}$, since they can be expressed as certain $r \times r$ minors of $M(\mathbf{x})$.

These polynomial solutions satisfy
\[\sum_i \sum_k c_{ijk}x_iy_k^{(p)}(\mathbf{x})=\Delta_{j,p}(\mathbf{x})\]
identically in $\mathbf{x}$ where $\Delta_{j,p}$ is a certain $(r+1) \times (r+1)$-minor of $M(\mathbf{x})$.

Differentiating this identity with respect to $x_{\nu}$, then multiplying by $y_j^{(q)}$ and summing over $j$ we obtain
\begin{equation}
    \label{eq.outcome_detumformung}
    \sum_j\sum_k c_{\nu jk} y_k^{(p)}y_j^{(q)}+\sum_k \Delta_{k,q} \frac{\partial y_k^{(p)}}{\partial x_{\nu}}=\sum_j y_j^{(q)} \frac{\partial \Delta_{j,p}}{\partial x_{\nu}}
\end{equation}
for all $1 \le \nu \le n$ and $1 \le p,q \le n-r$.

Now, since all $\Delta_{k,p}$ vanish on more than $H^{n-14+r+\varepsilon}$ points $\mathbf{x}$ with $\vert \mathbf{x}\vert <H$ it follows that the dimension of the variety described by the vanishing of all these determinants must be at least $n-13+r$. In particular, there must be a point $\mathbf{x}$ where the rank of the Jacobian matrix of these derivatives is at most $13-r$.

We now choose $\mathbf{x}$ as such a point. This means there are numbers $W_{j,p,\tau}$ and $U_{\tau,v}$ such that
\[\frac{\partial \Delta_{j,p}}{\partial x_{\nu}}=\sum_{\tau=1}^{13-r} W_{j,p,\tau} U_{\tau,\nu}.\]
The equation \eqref{eq.outcome_detumformung} now becomes
\[\sum_j \sum_k c_{\nu jk} y_k^{(p)} y_j^{(q)}=\sum_j y_j^{(q)} \sum_{\tau=1}^{13-r} W_{j,p,\tau} U_{\tau,\nu}.\]
Writing
\[\mathbf{Y}=T_1\mathbf{y}^{(1)}+\dots+T_{n-r}\mathbf{y}^{(n-r)}\]
for indeterminates $T_1,\dots, T_{n-r}$, multiplying the previous display by $T_pT_q$ and summing over $q$ we end up with
\[\sum_{j,k} c_{\nu jk} Y_jY_k=\sum_j \sum_{p=1}^{n-r} \sum_{q=1}^{n-r} T_pT_qy_j^{(q)}\sum_{\tau=1}^{13-r} W_{j,p,\tau} U_{\tau,\nu}=\sum_{\tau=1}^{13-r} V_{\tau} U_{\tau,\nu}\]
for certain numbers $V_{\tau}$. Multiplying by $Y_{\nu}$ and summing over $\nu$ we find that
\[C(\mathbf{Y})=\sum_{\tau=1}^{13-r} V_{\tau} \sum_{p=1}^{n-r} T_p \sum_{\nu=1}^n y_{\nu}^{(p)} U_{\tau,\nu}.\]
Note that the interior double sum is a linear form in the $T_p$ for each $\tau$. If all of these $13-r$ linear forms vanish, we see that $C(\mathbf{Y})=0$. But this means that $C$ vanishes on a linear subspace of dimension $(n-r)-(13-r)=n-13$ as desired and hence $h(C) \le 13$.

Finally, we estimate the size of the smallest solution in this subspace. From their definition as $r \times r$-minors we have $\vert \mathbf{y}^{(p)}\vert \ll H^rM^r$. Moreover, the values $U_{\tau,\nu}$ can be chosen as values of $\frac{\partial \Delta_{j,p}}{\partial x_{\nu}}$ and hence are bounded by $M^{r+1}H^r$. The coefficients of the linear system for the $T_i$ are therefore bounded by $M^{2r+1}H^{2r}$.

By an application of Siegel's Lemma, such a system has a non-trivial solution with $T_i \ll (M^{2r+1}H^{2r})^{13-r}$. Finally, this means that there is a non-trivial solution $\mathbf{Y}$ of $C(\mathbf{Y})=0$ satisfying
\[Y \ll M^rH^r (M^{2r+1}H^{2r})^{13-r}.\]
It is now readily checked that this is bounded by $M^{97}H^{91}$ for all choices of $r$. Since $H \le M^{\psi}$ by assumption, the result follows.
\end{proof}
\end{lemma}
Summarizing, we may therefore assume that $C$ is $\infty$-good for the purpose of Theorems \ref{thm.h=14asymptotic} and \ref{thm.smallsolution_h=14} as well as the non-singular case of \ref{thm.smallsolution_hom}. In the general case of Theorem \ref{thm.smallsolution_hom}, we may suppose that $C$ is $\psi$-good for some suitable $\psi$ as otherwise Lemma \ref{lemma_gc_hinvariante} allows us to deduce the existence of a relatively small solution.

\section{The Major Arcs}



We begin our analysis with the contribution from the major arcs. As remarked in \cite{bde12}, using Poisson's Summation Formula instead of the more elementary Euler Summation formula yields a better error term. For the usual problem of establishing an asymptotic formula, this improvement is irrelevant, but for the uniform version it changes the resulting exponent significantly.

The result of that approximation is the following lemma which is proved during the proof of Lemma 5 in \cite{bde12}.

\begin{lemma}
\label{lemma_approx_major_arcs}
Suppose that $f \in \mathbb{Z}[X_1,\dots,X_n]$ is a polynomial of degree $d \ge 3$ such that none of the partial derivatives $\frac{\partial^2 f}{\partial X_i^2}$ vanish identically.

Let $\mathscr{C}=\prod_{i=1}^n [a_i,b_i]$ be a box and put $R_{\mathscr{C}}=\max_i \vert b_i-a_i\vert$. Let $\lambda \in \mathbb{R}$ and $\psi \in (0,1]$ be chosen so that $\vert \lambda \nabla f(\mathbf{x})\vert \le 1-\psi$ for all $\mathbf{x} \in \mathscr{C}$. Then
\[\sum_{\mathbf{x} \in \mathscr{C}} e(\lambda f(\mathbf{x}))=\int_{\mathscr{C}} e(\lambda f(\mathbf{t})) \mathrm{d}\mathbf{t}+\mathcal{O}_d\left(\psi^{-1}R_{\mathscr{C}}^{n-1}\right).\]
\end{lemma}

\medskip

We now wish to estimate the Weyl sum $S(\alpha)$, where $\alpha=\frac{a}{q}+\beta$ for some $q \le P$. Sorting the initial sum by congruence classes modulo $q$, we get
\begin{align*}
    S\left(\frac{a}{q}+\beta\right)
    &=\sum_{\mathbf{x} \in P\mathcal{B} \cap \mathbb{Z}^n} e\left(\left(\frac{a}{q}+\beta\right)\phi(\mathbf{x})\right)\\
    &=\sum_{\mathbf{r}(q)} e\left(\frac{a\phi(\mathbf{r})}{q}\right)\sum_{\mathbf{y} \in \mathbb{Z}^n: \mathbf{r}+q\mathbf{y} \in P\mathcal{B}} e\left(\beta \phi(\mathbf{r}+q\mathbf{y})\right).
\end{align*}
We continue by applying Lemma \ref{lemma_approx_major_arcs} with $f(\mathbf{y})=\phi(\mathbf{r}+q\mathbf{y})$, $\lambda=\beta$ and $\mathscr{C}=\{\mathbf{y}: \mathbf{r}+q\mathbf{y} \in P\mathcal{B}\}$, so that $R_{\mathscr{C}}= \frac{2P}{q}$. For $\vert \beta\vert \le u$, the bound on the derivative will be satisfied with $\psi=\frac{1}{2}$ as soon as
\begin{equation}
\tag{$\mathfrak{M}_2$}
u \cdot P_0 \cdot M \cdot \vert \mathbf{z}\vert^2 \ll P
\end{equation}
with a sufficiently small implicit constant.
This yields
\begin{align*}
    S\left(\frac{a}{q}+\beta\right)
    &=\sum_{\mathbf{r}(q)} e\left(\frac{a\phi(\mathbf{r})}{q}\right)\left(\int_{t: \mathbf{r}+q\mathbf{t} \in P\mathcal{B}} e\left(\beta \phi(\mathbf{r}+q\mathbf{t})\right) dt+\mathcal{O}\left(\frac{P^{n-1}}{q^{n-1}} \right)\right)\\
    &=\frac{S(q,a)}{q^n} \int_{P\mathcal{B}} e(\beta \phi(\mathbf{t})) dt+\mathcal{O}\left(P^{n-1}q\right),
\end{align*}
where we write
\[S(q,a)=\sum_{\mathbf{r}(q)} e\left(\frac{a\phi(\mathbf{r})}{q}\right).\]

Integrating this approximation over $\vert \beta\vert \le \frac{u}{P^3}$ and summing over $a$ and $q$, we obtain
\begin{align*}
    \int_{\mathfrak{M}} S(\alpha) d\alpha&=\mathfrak{S}(P_0) \int_{\vert \beta\vert \le \frac{u}{P^3}} \int_{P\mathcal{B}} e(\beta \phi(\mathbf{t})) d\mathbf{t}+\mathcal{O}\left( P^{n-4}P_0^3u\right)\\
    &=\mathfrak{S}(P_0) \cdot \int_{\vert \beta\vert \le u} \int_{\mathcal{B}} e\left(\beta \frac{\phi(P\mathbf{t})}{P^3}\right) d\mathbf{t} \cdot P^{n-3}+\mathcal{O}\left( P^{n-4}P_0^3u\right)\\
    &=\mathfrak{S}(P_0) \cdot \int_{\vert \beta\vert \le u} \int_{\mathcal{B}} e\left(\beta C(\mathbf{t})+\mathcal{O}\left(\frac{uM\vert \mathbf{z}\vert^2}{P}\right)\right) d\mathbf{t} \cdot P^{n-3}+\mathcal{O}\left( P^{n-4}P_0^3u\right)\\
    &=\mathfrak{S}(P_0) \cdot \left(\mathfrak{I}(u)+\mathcal{O}\left(\frac{u^2M\vert \mathbf{z}\vert^2}{P}\right)\right) \cdot P^{n-3}+\mathcal{O}\left( P^{n-4}P_0^3u\right),
\end{align*}
with
\[\mathfrak{S}(P_0):=\sum_{q \le P_0} \sum_{(a;q)=1} \frac{S(q,a)}{q^n}\]
and
\[\mathfrak{I}(u):=\int_{\vert \beta\vert \le u} \int_{\mathcal{B}} e(\beta C(\mathbf{t})) d\mathbf{t}.\]


\section{The singular integral}

We now need to estimate the singular integral $\mathfrak{I}(u)$. As usual, this is done by first choosing the center of the box $\mathcal{B}$ to be a suitable non-singular point and then using Fourier's Inversion Theorem to show that $\mathfrak{I}(u)$ converges to a positive number $\mathfrak{I}$ as $u \to \infty$.

However, to obtain the desired uniform result, we also require a bound on the difference $\vert \mathfrak{I}-\mathfrak{I}(u)\vert$. In \cite{bde12}, a uniform version of Fourier's Inversion Theorem was cited from the thesis of Lloyd~\cite{lloyd}. Since Lloyd's Thesis is not publically available and there actually was a small mistake in the application of his result, we decided to include a self-contained treatment of the singular integral, closely following the original argument of Lloyd.

The key technical result is the following:

\begin{lemma}
\label{lemma_sing_int_conv}
Suppose that $n \ge 2$ and that $\mathcal{B}$ is a box with center $\mathbf{z}$ and of width $\rho$ and suppose that $C$ satisfies
\[\frac{\partial C}{\partial x_i} \ge \partial_i\]
for $i=1,2$ on all of $\mathcal{B}$ for some positive constants $\partial_1,\partial_2$. Then, for
\[\mathfrak{I}(Z):=\int_{\vert \beta\vert \le Z} \int_{\mathcal{B}} e(\beta C(\mathbf{x})) d\mathbf{x} dt,\]
one has
\[\mathfrak{I}(Z)=V(0) \cdot \left(1+\mathcal{O}\left(\frac{1}{\sigma Z}\right)\right)+\mathcal{O}\left(\frac{\rho^{n-2}}{Z}\left(\frac{1}{\partial_1\partial_2}+\frac{\rho\vert \mathbf{z}\vert M}{\partial_1^3} \log(\sigma Z)\right)\right)\]
whenever $\sigma Z \ge 2$, where $\sigma=\max_{\mathbf{x} \in \mathcal{B}} \vert C(\mathbf{x})\vert$.

We have the bound
\begin{equation}
    \label{eq.boundV(t)}
    V(0) \gg \frac{\rho^{n-1}}{M\vert \mathbf{z}\vert^2}.
\end{equation}
\end{lemma}
Some remarks are in order: Compared to Lemma 7 in \cite{bde12}, our last error term is better. Indeed, as we will explain, it is the precise outcome of Lloyd's argument. It has the additional virtue of being scaling invariant: If we replace $\mathcal{B}$ by $T\mathcal{B}$ and $Z$ by $\frac{Z}{T^3}$, then $\mathfrak{I}(Z)$ is multiplied by $T^{n-3}$ and the same is true for all error terms (in contrast to the version in \cite{bde12}).

Note that we have also refrained from writing the first term as $V(0)+\mathcal{O}\left(\frac{V(0)}{\sigma Z}\right)$. This is convenient because it means that we do not require an upper bound for $V(0)$ in order to deduce a lower bound for $\mathfrak{I}(Z)$.

Finally, note that the condition $\frac{\partial C}{\partial x_i} \ge \partial_i$ on all of $\mathcal{B}$ leads to a restriction on the size of $\rho$ that was overlooked (and indeed not satisfied by the choice made) in \cite{bde12}.

\begin{proof}
We begin by interchanging the order of integration to write
\[\mathfrak{I}(Z)=\int_{\mathcal{B}} \frac{\sin 2\pi ZC(\mathbf{x})}{\pi C(\mathbf{x})} d\mathbf{x}.\]
The next step is to make the change of variables $(x_1,\dots,x_n) \mapsto (t,x_2,\dots,x_n)$ with $t=C(\mathbf{x})$. The Jacobian of this change of variable is given by $\frac{\partial C}{\partial x_1} \ge \partial_1>0$ and so this change of variables is invertible. If we denote by $g(t,x_2,\dots,x_n)$ the coordinate $x_1$ of the inverse, then $0<\frac{\partial g}{\partial t}=\frac{1}{\frac{\partial C}{\partial x_1}} \le \frac{1}{\partial_1}$ and we can write
\[\mathfrak{I}(Z)=\int_{\mathcal{R}} \frac{\sin 2\pi Zt}{\pi t} \frac{\partial g}{\partial t}(t,x_2,\dots,x_n) dt dx_2\dots dx_n=\int_{-\sigma}^{\sigma} \frac{\sin 2\pi Zt}{\pi t} V(t) dt\]
where
\begin{equation}
\label{eq.V(t)}
    V(t)=\int_{(t,x_2,\dots,x_n) \in \mathcal{R}} \frac{\partial g}{\partial t}(t,x_2,\dots,x_n) dx_2\dots dx_n
\end{equation}
and $\mathcal{R}$ is the image of $\mathcal{B}$ under the change of variables.

For later use, we also record the lower bound
\begin{equation}
    \label{eq.bound_g1}
    \frac{\partial g}{\partial t}=\frac{1}{\frac{\partial C}{\partial x_1}} \gg \frac{1}{M\vert \mathbf{z}\vert^2}.
\end{equation}

To make use of Fourier's Inversion Theorem, we need to estimate right and left derivatives of the function $V(t)$.

To this end, we write
\[V(t)=\int_{\prod_{i=3}^n [z_i-\rho,z_i+\rho]} \int_{x_2 \in \mathcal{R}_{t,x_3,\dots,x_n}} \frac{\partial g}{\partial t}(t,x_2,\dots,x_n) dx_2 dx_3\dots dx_n\]
where
\begin{align*} 
\mathcal{R}_{t,x_3,\dots,x_n}&=\{x_2 \in (z_2-\rho,z_2+\rho): \exists x_1 \in (z_1-\rho,z_1+\rho): t=C(\mathbf{x})\}\\
&= \{x_2 \in (z_2-\rho, z_2+\rho): C(z_1-\rho,x_2,\dots,x_n)<t<C(z_1+\rho,x_2,\dots,x_n)\}\\
&=\{x_2 \in (z_2-\rho, z_2+\rho): b^{(1)}_{x_3,\dots,x_n}(x_2)<t<b^{(2)}_{x_3,\dots,x_n}(x_2)\}
\end{align*}
with $b^{(1)}$ and $b^{(2)}$  defined by the last equation so that we have $\frac{\partial b^{(i)}}{\partial x_2} \ge \partial_2$ by assumption. In particular, we can write
\[\mathcal{R}_{t,x_3,\dots,x_n}=(\ell^{(1)}_{x_3,\dots,x_n},\ell^{(2)}_{x_3,\dots,x_n})\]
with $\ell^{(i)}_{x_3,\dots,x_n}$ continuous everywhere and continuously differentiable with the exception of at most two points. At these two points, left and right derivatives exist and all of these derivatives satisfy $0<\frac{\partial \ell^{(i)}}{\partial t} \le \frac{1}{\partial_2}$. We now obtain
\[V(t)=\int_{\prod_{i=3}^n [z_i-\rho,z_i+\rho]} \int_{\ell^{(1)}}^{\ell^{(2)}} \frac{\partial g}{\partial t}(t,x_2,\dots,x_n) dx_2 dx_3\dots dx_n.\]
Using Leibniz's rule, we now obtain that $V$ has right and left derivatives everywhere with them disagreeing only at finitely many points. More precisely, the left and right derivatives of the inner integral are given by a linear combination of expressions of the form $\frac{\partial g}{\partial t} \cdot \frac{\partial^{\pm} \ell^{(i)}}{\partial t}$ as well as
\[\int_{\ell^{(1)}}^{\ell^{(2)}} \frac{\partial^2 g}{\partial t^2} dt.\]
The first type of expressions is bounded by $\mathcal{O}\left(\frac{1}{\partial_1\partial_2}\right)$ and the second one by $\frac{\rho M \vert \mathbf{z}\vert}{\partial_1^3}$ on noting that
\[\frac{\partial^2 g}{\partial t^2}=-\frac{\frac{\partial^2 C}{\partial x_1}}{\left(\frac{\partial C}{\partial x_1}\right)^3} \ll \frac{M \vert \mathbf{z}\vert}{\partial_1^3}\]
(where we used that $\rho \ll \vert \mathbf{z}\vert$ and hence $\vert \mathbf{z}\vert+\rho \ll \vert \mathbf{z}\vert$ since the box clearly can't contain the origin). It now follows that
\[\frac{\partial^{\pm} V}{\partial t} \ll \rho^{n-2} \cdot \left(\frac{1}{\partial_1\partial_2}+\frac{\rho M\vert \mathbf{z}\vert}{\partial_1^3}\right)=:A.\]
Letting 
\[\Phi(t)=\frac{V(t)+V(-t)-2V(0)}{t},\]
we now see that $\Phi(t) \ll A$ and $\Phi'(t) \ll \frac{A}{t}$ for all $t>0$.

Finally, we are ready to estimate
\[\mathfrak{I}(Z)=2V(0) \int_0^{\sigma} \frac{\sin 2\pi Zt}{\pi t} dt+\frac{1}{\pi} \int_0^{\sigma} \Phi(t) \sin 2\pi Zt dt.\]
The first integral is easily evaluated to $\frac{1}{2}+\mathcal{O}\left(\frac{1}{\sigma Z}\right)$. For the second integral, we split the range of integration into $t \le \tau$ and $t \ge \tau$ for a suitable parameter $0 \le \tau \le \sigma$. The range $t \le \tau$ contributes $\ll \tau A$. On the range $t \ge \tau$ we can integrate by parts to obtain
\[\int_{\tau}^{\sigma} \Phi(t) \sin 2\pi Zt dt=\left[-\Phi(t) \frac{\cos 2\pi Zt}{2\pi t}\right]_\tau^{\sigma}+\int_{\tau}^{\sigma} \Phi'(t) \frac{\cos 2\pi Zt}{2\pi t} dt \ll \frac{A}{Z}\left(1+\log \frac{\sigma}{\tau}\right).\]
The main result of the lemma now follows upon choosing $\tau=\frac{1}{Z}$.

Finally, we note that \eqref{eq.boundV(t)} follows immediately from \eqref{eq.V(t)} and \eqref{eq.bound_g1}.
\end{proof}

To apply Lemma \ref{lemma_sing_int_conv}, we now need to choose a suitable non-singular point $\mathbf{z}$ as the center of our box. The strategy is similar to the one in the proof of Lemma 6 in \cite{bde12}, but we need a variant for the inhomogeneous case.

\begin{lemma}
\label{lemma_good_nonsingular_point}
a) If $h=h(C)$, there is a solution $\widetilde{\mathbf{z}}=(\xi,\mathbf{y}) \in \mathbb{R}^n$ to $C(\mathbf{z})=0$ satisfying $\vert \mathbf{z}\vert \ll M^{\frac{1}{h-2}}$ and such that (possibly after relabeling the coordinates and changing signs)
\[\frac{\partial C(\widetilde{\mathbf{z}})}{\partial x_1} \gg M^{-1-\frac{4}{h-2}}\]
and
\[\frac{\partial C(\widetilde{\mathbf{z}})}{\partial x_2} \gg M^{-2-\frac{7}{h-2}}.\]
b) Similarly, unless $C(\mathbf{x})=0$ has a non-trivial integer solution with $\vert \mathbf{x}\vert \ll M^{\frac{1}{n-2}}$, there is a solution $\widetilde{\mathbf{z}}=(\xi,\mathbf{y}) \in \mathbb{R}^n$ to $C(\widetilde{\mathbf{z}})=0$ satisfying $\vert \mathbf{z}\vert \ll M^{\frac{1}{n-2}}$ and such that (possibly after relabeling the coordinates and changing signs)
\[\frac{\partial C(\widetilde{\mathbf{z}})}{\partial x_1} \gg M^{-1-\frac{4}{n-2}}\]
and
\[\frac{\partial C(\widetilde{\mathbf{z}})}{\partial x_2} \gg M^{-2-\frac{7}{n-2}}.\]
\end{lemma}
\begin{proof}
We write the cubic form as
\[C(\mathbf{x})=ax_1^3+F_1x_1^2+F_2x_1+F_3\]
where as explained in the introduction we may assume that $a>0$ and $a \gg M$. In the setting of b), by Siegel's Lemma, we can find a non-trivial integer solution $\mathbf{y}$ to $F_1(\mathbf{y})=0$ with $\vert \mathbf{y}\vert \ll M^{\frac{1}{n-2}}$. If $F_3(\mathbf{y})=0$, we have found the desired small integer solution $(0,\mathbf{y})$. Otherwise we may assume that $\vert F_3(\mathbf{y})\vert \ge 1$.

In the setting of a), we argue instead that we can find linearly independent integer solutions $\mathbf{y}^{(1)}, \dots, \mathbf{y}^{(n-h+1)}$ of $F_1(\mathbf{y})=0$, all of them satisfying $\mathbf{y}^{(i)} \ll M^{\frac{1}{h-2}}$. By definition of $h=h(C)$, one of them must have $F_3(\mathbf{y}) \ne 0$ and hence $\vert F_3(\mathbf{y})\vert \ge 1$.

From here on, the argument is identical in both cases, so we only treat a).

Flipping signs if necessary, we may assume that $F_3=F_3(\mathbf{y}) \le -1$. We may thus find a real zero $\xi>0$ of $C(\xi,\mathbf{y})=a\xi^3+F_2\xi+F_3=0$ where we have written $F_2=F_2(\mathbf{y})$. The next step is to establish bounds on $\xi$.

If $F_2 \ge 0$, we can use that $a\xi^3 \le \vert F_3\vert \ll M^{1+\frac{3}{h-2}}$, so that $\xi \ll M^{\frac{1}{h-2}}$ and then
\[\xi \ge \frac{1}{a\xi^2+F_2} \gg M^{-1-\frac{2}{h-2}}.\]
If $F_2<0$, we instead argue that $a\xi^3 \ge 1$, so that $\xi \gg M^{-1/3}$,  $a\xi^3 = \vert F_2\vert \xi+\vert F_3\vert$ and
\[\xi \ll  \left \vert \frac{F_2}{a}\right\vert^{1/2}+\left\vert \frac{F_3}{a}\right\vert^{1/3} \ll M^{\frac{1}{h-2}}.\]
In any case, we have thus established that
\[M^{-1-\frac{2}{h-2}} \ll \xi \ll M^{\frac{1}{h-2}}.\]
Finally, we need to bound the partial derivatives. We have
\[\frac{\partial C(\xi,\mathbf{y})}{\partial x_1}=3a\xi^2+F_2=2a\xi^2-\frac{F_3}{\xi} \ge 2a\xi^2 \gg M^{-1-\frac{4}{h-2}}.\]
Using Euler's identity, we now find that
\[\left\vert y_2 \frac{\partial C}{\partial x_2}+\dots+y_n \frac{\partial C}{\partial x_n}\right\vert \ge \left\vert \xi \frac{\partial C}{\partial x_1}\right\vert -3\vert C(\xi,\mathbf{y})\vert \gg M^{-2-\frac{6}{h-2}}\]
and hence w.l.o.g. $\left\vert y_2 \frac{\partial C}{\partial x_2}\right\vert \gg M^{-2-\frac{6}{h-2}}$, so that $\left\vert \frac{\partial C}{\partial x_2}\right\vert \gg M^{-2-\frac{7}{h-2}}$ as desired.
\end{proof}
We now choose our box $\mathcal{B}$ with center $\mathbf{z}$ and width $\rho=1$ making sure that the assumptions in Lemma \ref{lemma_sing_int_conv} are satisfied. To this end, we choose $\mathbf{z}=AM^{3+\frac{8}{n-2}} \widetilde{\mathbf{z}}$ with $\widetilde{\mathbf{z}}$ as in Lemma \ref{lemma_good_nonsingular_point} and a sufficiently large constant $A>0$. With the choice of $h=14$ or $n=14$, respectively, we record the properties of this choice in the following lemma.

\begin{lemma}
\label{lemma_center_B}
The point $\mathbf{z} \in \mathbb{R}^n$ is a solution of $C(\mathbf{z})=0$ satisfying $\vert \mathbf{z}\vert \ll M^{3.75}$,
\[\frac{\partial C}{\partial x_1} \gg M^{6}\]
and
\[\frac{\partial C}{\partial x_2} \gg M^{4.75}\]
on all of $\mathcal{B}$.
\end{lemma}
For the proof, we only need to note that the bounds for the derivatives at the point $\mathbf{z}$ (which are obtained directly from the bounds for $\widetilde{\mathbf{z}}$ by scaling) extend over all of $\mathcal{B}$ as
\[\frac{\partial C}{\partial x_i}=\frac{\partial C(\mathbf{z})}{\partial x_i}+\mathcal{O}(M\vert \mathbf{z}\vert).\]
It is here that we make use of the fact that $A$ is sufficiently large.

It is clear that we may assume that $\frac{\partial C}{\partial x_i} \ll \frac{\partial C}{\partial x_1}$ for all $i$ at $\mathbf{z}$ and then on all of $\mathcal{B}$ as otherwise we can simply permute the variables.

Finally, we can collect the results of this section in the following lemma.

\begin{lemma}
\label{lem.majorarcsummary_vorsingulaerereihe}
With the box $\mathcal{B}=\mathcal{B}(\mathbf{z})$ chosen as above, we have
\[\mathfrak{I}(u)=\mathfrak{I} \cdot \left(1+\mathcal{O}\left(\frac{1}{M^{12.25} u}\right)\right)+\mathcal{O}\left(\frac{1}{uM^{10.75}}\right)\]
for some number $\mathfrak{I}>0$ with
\[\mathfrak{I} \gg \frac{1}{M^{8.5}}.\]
In particular, we have
\[\mathfrak{I}(u) \gg \frac{1}{M^{8.5}}\]
for any  $u \ge 1$.

Under the assumption of $(\mathfrak{M}_1)$ and $(\mathfrak{M}_2)$ as well as
\begin{equation}
    \tag{$\mathfrak{I}_1$}
    u^2M^{17+\varepsilon} \ll P,
\end{equation}
it follows that
\[\int_{\mathfrak{M}} S(\alpha) d\alpha =(1+o(1))\mathfrak{S}(P_0) \cdot \mathfrak{I}(u) \cdot P^{n-3}+\mathcal{O}\left( P^{n-4}P_0^3u\right).\]
\end{lemma}

To discuss the singular series $\mathfrak{S}(P_0)$, we need estimates on the Gauß Sums $S(q,a)$, a problem which is related to bounding the Weyl Sum $S(\alpha)$ on the minor arcs. We therefore continue with the discussion of the minor arcs and return to discuss the singular series at the appropriate point.

\section{The minor arcs}

\label{sec.minorarcs}

As in Heath-Brown's work~\cite{heath2007cubic}, we make use of three different methods to bound $S(\alpha)$ on the minor arcs. First of all, there is the classical Weyl differencing method, that was used by Davenport to obtain his result for $16$ variables. While alone it is therefore insufficient for our purposes, it still outperforms the other methods in certain regimes of the minor arcs and therefore remains a crucial ingredient.

The second method is the pointwise van der Corput differencing. This improves on the Weyl differencing, but again is not good enough to save more variables. We shall only require it to bound the Gauß Sums $S(q,a)$ and thus for the convergence of the singular series.

Finally, the third method is the mean-square average version of van der Corput differencing which is the key innovation of Heath-Brown in allowing us to obtain results for $14$ variables.

Since in all methods, lower order terms of $\phi$ disappear in the course of differencing, the results of this section are essentially identical with those from the homogeneous case.

\subsection{Preliminaries}

We begin by recalling the general strategy implicit already in Davenport's work:

By an appropriate combination of squaring and Cauchy-Schwarz, one reduces the cubic exponential sum to one over a linear form. While in the classical case of a diagonal cubic form, the resulting sum is easy to handle, the general shape of a cubic form begins to cause problems.

In general, this step allows us to reduce a bound for $S(\alpha)$ to one for the number of solutions to a system of certain auxiliary diophantine inequalities involving the bilinear forms $B_i(\mathbf{x},\mathbf{y})$.

These diophantine inequalities are dealt with by an application of Davenport's Shrinking Lemma. Roughly speaking, the Shrinking Lemma allows us to bootstrap the diophantine inequalities in a way that forces equality.

Finally, the number of solutions to the resulting system of auxililary diophantine equations involving the bilinear forms $B_i(\mathbf{x},\mathbf{y})$ can be estimated using Davenport's Geometric Condition which for us is captured by the assumption \eqref{eq.good_form} that $C$ is $\psi$-good.

At this point, we record two of the mentioned key ingredients. The first one is Davenport's Shrinking Lemma (see e.g. \cite{heath2007cubic}, Lemma 2.2).

\begin{lemma}
\label{lem.shrinkinglemma}
\label{geoz}
Let $L \in M_n(\mathbb{R})$ be a real symmetric $n \times n$ matrix. Let $a>0$ be real, and let
\[N(Z):=\#\{\mathbf{u} \in \mathbb{Z}^n: \vert\mathbf{u}\vert \le aZ, \|(L\mathbf{u})_i\|<a^{-1}Z, 1 \le i \le n\}.\]
Then if $0<Z \le 1$, we have
\[N(1) \ll_n Z^{-n} N(Z).\]
\end{lemma}

The second one is Lemma 2.3 from \cite{heath2007cubic} and will allow us to deduce that a sufficiently strong diophantine inequality already forces equality or at least a divisibility condition:

\begin{lemma}
\label{lem.bootstrap}
Let a real number $X \ge 0$ be given and let $\alpha=\frac{a}{q}+\theta$ with $(a;q)=1$ and $2qX\vert\theta\vert \le 1$. Suppose that $m \in \mathbb{Z}$ is such that $\vert m\vert \le X$ and $\|\alpha m\| \le \frac{1}{P_1}$ for some $P_1 \ge 2q$. Then $q \mid m$. In particular we will have $m=0$ if in addition $X<q$ or $\vert \theta\vert>\frac{1}{qP_1}$.
\end{lemma}

\subsection{Weyl Differencing}

Recall the definition
\[S(\alpha)=\sum_{\mathbf{x} \in P\mathscr{B}} e(\alpha \phi(\mathbf{x})).\]
This leads to the identity
\[\vert S(\alpha)\vert^2=\sum_{\mathbf{x},\mathbf{y} \in P\mathscr{B}} e(\alpha(\phi(\mathbf{y})-\phi(\mathbf{x}))).\]
Writing $\mathbf{y}=\mathbf{x}+\mathbf{d_1}$ we can rewrite this as
\[\vert S(\alpha)\vert^2=\sum_{\mathbf{d_1}} \sum_{\mathbf{x} \in \mathcal{R}(\mathbf{d_1})} e(\alpha(\phi(\mathbf{x}+\mathbf{d_1})-\phi(\mathbf{x}))),\]
where $\mathcal{R}(\mathbf{d_1})=P\mathscr{B} \cap (P\mathscr{B}-\mathbf{d_1})$. In particular, the inner sum is empty unless $\vert \mathbf{d_1}\vert \le 2 P$. Squaring again and applying Cauchy-Schwarz, we then find that
\[\vert S(\alpha)\vert^4 \ll P^n \sum_{\mathbf{d_1}} \sum_{\mathbf{x},\mathbf{z} \in \mathcal{R}(\mathbf{d_1})} e(\alpha(\phi(\mathbf{z}+\mathbf{d_1})-\phi(\mathbf{z})-\phi(\mathbf{x}+\mathbf{d_1})+\phi(\mathbf{x}))).\]
Writing $\mathbf{z}=\mathbf{x}+\mathbf{d_2}$ this can be rewritten as
\begin{equation}\label{weyl}\vert S(\alpha)\vert^4 \ll P^n \sum_{\mathbf{d_1}, \mathbf{d_2}} \sum_{\mathbf{x} \in \mathcal{S}(\mathbf{d_1},\mathbf{d_2})} e(\alpha C(\mathbf{d_1},\mathbf{d_2},\mathbf{x}))\end{equation}
where $\mathcal{S}(\mathbf{d_1},\mathbf{d_2})=\mathcal{R}(\mathbf{d_1}) \cap \left(\mathcal{R}(\mathbf{d_1})-\mathbf{d_2}\right)$ and
\[C(\mathbf{d_1},\mathbf{d_2},\mathbf{x})=\phi(\mathbf{x}+\mathbf{d_1}+\mathbf{d_2})-\phi(\mathbf{x}+\mathbf{d_1})-\phi(\mathbf{x}+\mathbf{d_2})+\phi(\mathbf{x}).\]
Note that the notation $C(\mathbf{d_1},\mathbf{d_2}, \mathbf{x})$ is appropriate as the lower order terms of $\phi$ have disappeared at this point.

All we need to know about $\mathcal{S}(\mathbf{d_1},\mathbf{d_2})$ is that it is a certain box inside $P\mathscr{B}$. Further, note that
\[C(\mathbf{d_1},\mathbf{d_2},\mathbf{x})=6\sum_{i=1}^n x_iB_i(\mathbf{d_1},\mathbf{d_2})+\psi(\mathbf{d_1},\mathbf{d_2}),\]
where $\psi(\mathbf{d_1},\mathbf{d_2})$ is independent of $\mathbf{x}$. We now recall the standard bound for linear exponential sums
\[\sum_{x \in I} e(\alpha x) \ll \min(\vert I\vert, \|\alpha\|^{-1}),\]
where $I \subset \mathbb{R}$ is any interval and $\|\alpha\|=\min_{n \in \mathbb{Z}} \vert \alpha-n\vert$. From this and the previous discussion, it now follows that
\begin{equation}\label{weyl2}\vert S(\alpha)\vert^4 \ll P^n \sum_{\vert \mathbf{d_1}\vert,\vert\mathbf{d_2}\vert \le 2P} \prod_{i=1}^n \min\left( P, \|6\alpha B_i(\mathbf{d_1},\mathbf{d_2})\|^{-1}\right).\end{equation}
The next step is to compute the sum over $\mathbf{d_1}$ and $\mathbf{d_2}$ or rather relate it to the previously mentioned number of solutions to a certain system of inequalities. To this end, let
\[N(\mathbf{d})=\#\{\mathbf{x} \in \mathbf{Z}^n: \vert \mathbf{d}\vert \le 4P, \|6\alpha B_i(\mathbf{d},\mathbf{x})\| < \frac{1}{4P}\}.\]
It then follows that for fixed $\mathbf{d_1}$ and integers $r_1,\dots,r_n$ with $0 \le r_i<4P$, there are at most $N(\mathbf{d_1})$ values of $\mathbf{d_2}$ with $\vert \mathbf{d_2}\vert \le 2 P$ satisfying
\[\frac{r_i}{4P} \le \{6\alpha B_i(\mathbf{d_1},\mathbf{d_2})\}<\frac{r_i+1}{4P},\]
because for any two such vectors $\mathbf{d_2}$ and $\mathbf{d_2}'$ their difference $\mathbf{d}=\mathbf{d_2}-\mathbf{d_2'}$ must be in the set counted by $N(\mathbf{d_1})$. This yields the estimate
\begin{align*}
\vert S(\alpha)\vert^4 &\ll P^n \sum_{\mathbf{d_1}}N(\mathbf{d_1}) \sum_{r_1=0}^{ 4P} \dots \sum_{r_n=0}^{4P} \prod_{i=1}^n \min\left(P, \frac{4P}{r_i}\right)\\
&\ll P^{2n} (\log P)^n \sum_{\mathbf{d_1}} N(\mathbf{d_1})
\end{align*}
so that
\[\vert S(\alpha)\vert^4 \ll P^{2n+\varepsilon} \# \left\{(\mathbf{x},\mathbf{y}) \in \mathbb{Z}^{2n}: \vert \mathbf{x}\vert, \vert \mathbf{y}\vert \le 4P, \|6\alpha B_i(\mathbf{x},\mathbf{y})\|<\frac{1}{4P}\right\}.\]

An application of the Shrinking Lemma \ref{lem.shrinkinglemma} now leads to the estimate
\[\vert S(\alpha)\vert^4 \ll Z^{-n}P^{2n+\varepsilon} \# \left\{(\mathbf{x},\mathbf{y}) \in \mathbb{Z}^{2n}: \vert \mathbf{x}\vert \le 4P, \vert \mathbf{y}\vert \le 4 PZ, \|6\alpha B_i(\mathbf{x},\mathbf{y})\|<\frac{Z}{4P}\right\}.\]
Reversing the rôles of $\mathbf{x}$ and $\mathbf{y}$ and applying the argument again with slightly different parameters, we arrive at
\begin{equation}
\label{eq.weyl_shrinking_outcome}\vert S(\alpha)\vert^4 \ll Z^{-2n}P^{2n+\varepsilon} \# \left\{(\mathbf{x},\mathbf{y}) \in \mathbb{Z}^{2n}: \vert \mathbf{x}\vert, \vert \mathbf{y}\vert \le 4PZ, \|6\alpha B_i(\mathbf{x},\mathbf{y})\|<\frac{Z^2}{4P}\right\}.
\end{equation}

We now need to choose $Z$ sufficiently small so that Lemma \ref{lem.bootstrap} allows us to conclude $B_i(\mathbf{x},\mathbf{y})=0$.

Here we choose $m=6B_i(\mathbf{x},\mathbf{y})$ so that $X \asymp M(PZ)^2$ and $P_1 \asymp \frac{P}{Z^2}$. Thus any choice of $Z$ satisfying $Z \le 1$ as well as 
\[2q\vert\theta\vert M(PZ)^2 \ll 1\quad \text{ and } \quad Z^2q \ll P\]
as well as
\[M(PZ)^2 \ll q \quad \text{ or } \quad Z^2  \ll \vert\theta\vert qP\]
with sufficiently small implicit constants allows us to conclude that
\begin{equation}\label{weyl4}\vert S(\alpha)\vert^4 \ll Z^{-2n}P^{2n+\varepsilon}\# \left\{(\mathbf{x},\mathbf{y}) \in \mathbb{Z}^{2n}: \vert \mathbf{x}\vert, \vert \mathbf{y}\vert \le 4 PZ, B_i(\mathbf{x},\mathbf{y})=0\right\}.\end{equation}
In the end we will choose $Z$ as big as possible, subject to the conditions above, but before making this choice let us see how to estimate the RHS in (\ref{weyl4}). At this point we need the assumption that $C$ is $\psi$-good. Recall that this means that the estimate
\begin{equation}\label{againgood}\#\{\mathbf{x} \in \mathbb{Z}^n: \vert \mathbf{x}\vert<H, r(\mathbf{x})=r\} \ll H^{n-14+r+\varepsilon}\end{equation}
holds uniformly in $1 \le H \le M^{\psi}$ where $r(\mathbf{x})$ is the rank of the matrix $M(\mathbf{x})$. This is clearly related to the system of equations we are studying by the fact that the condition $B_i(\mathbf{x},\mathbf{y})=0$ for all $i$ is equivalent to $M(\mathbf{y})\mathbf{x}=\mathbf{0}$.

But if $\mathbf{y}$ is fixed, the number of $\mathbf{x}$ with $\vert\mathbf{x}\vert \le 4 P Z$ and $M(\mathbf{y})\mathbf{x}=0$ is $\mathcal{O}\left((ZP)^{n-r}\right)$. Hence with $H\asymp PZ$ we find that
\begin{align*}
\vert S(\alpha)\vert^4 &\ll Z^{-2n}P^{2n+\varepsilon} \sum_{r=0}^n \sum_{\vert \mathbf{w}\vert \ll H, r(\mathbf{w})=r} (PZ)^{n-r}\\
&\ll Z^{-2n} P^{2n+\varepsilon} \sum_{r=0}^{14} H^{n-14+r+\varepsilon} (PZ)^{n-r}\\
&\ll P^{4n+\varepsilon}(PZ)^{-14},
\end{align*}
assuming $PZ \ll M^{\psi}$. Here we needed to assume that $PZ \gg 1$ but the final result is trivially true if this assumption fails to be correct. Finally, from this result it is clear that the optimal choice for $Z$ is indeed the maximal one subject to the above conditions. This choice is
\[Z \asymp \min\left(1,\frac{1}{(q\vert\theta\vert M)^{\frac{1}{2}} P}, \left(\frac{P}{q}\right)^{\frac{1}{2}}, \frac{M^{\psi}}{ P}, \max\left(\frac{q^{\frac{1}{2}}}{M^{\frac{1}{2}} P}, (qP\vert \theta\vert)^{\frac{1}{2}}\right)\right)\]
and leads to the following final result.

\begin{lemma}
\label{lem.weylinequality}
Assume that $C$ is $\psi$-good. If $\alpha=\frac{a}{q}+\theta$ for coprime integers $0 \le a<q$, then
\[S(\alpha) \ll P^{n+\varepsilon} \left(\frac{1}{P^2}+Mq\vert \theta\vert+\frac{q}{P^3}+\frac{1}{q}\min\left(M, \frac{1}{\vert\theta\vert P^3}\right)+ M^{-2\psi}\right)^{\frac{7}{4}}.\]
\end{lemma}

\subsection{A pointwise bound via van der Corput}

\label{Saqsection}

In this section, we derive a bound for $S(q,a)$ using the version of van der Corput's method from \cite{heath2007cubic} instead of Weyl differencing. To this end, we will temporarily put $\mathscr{B}=(0,1]^n, P=q$ and $\alpha=\frac{a}{q}$ with $(a;q)=1$, so that $S(\alpha)=S(q,a)$. Of course the arguments in this section can be developed in a much broader context (see \cite{heath2007cubic} for more details), but since in their rôle as bounds for $S(\alpha)$ on the minor arcs they will be insufficient and in fact superseded by the results from the next section, we content ourselves with the treatment of this special case.
The basic idea is to write
\[S(q,a)=H^{-n}\sum_{\mathbf{h} \le H} \sum_{\mathbf{x}: \mathbf{x}+\mathbf{h} \le q} e_q(a \phi(\mathbf{x}+\mathbf{h})),\]
where $1 \le H \le q$ is a suitable parameter. Interchanging the order of summation, this yields
\[S(q,a)=H^{-n}\sum_{\mathbf{x} \in \mathbb{Z}^n} \sum_{\substack{\mathbf{h} \le H:\\ \mathbf{x}+\mathbf{h} \le q}} e_q(a \phi(\mathbf{x}+\mathbf{h})).\]
Note that the inner sum is non-empty only for $\mathcal{O}(q^n)$ vectors $\mathbf{x}$, due to the condition $H \le q$. Hence an application of Cauchy-Schwarz leads to
\[\vert S(q,a)\vert^2 \ll H^{-2n} q^n\sum_{\mathbf{x} \in \mathbb{Z}^n} \left\vert \sum_{\substack{\mathbf{h} \le H:\\ \mathbf{x}+\mathbf{h} \le q}} e_q(a \phi(\mathbf{x}+\mathbf{h}))\right\vert^2.\]
Opening the square, this yields
\[\vert S(q,a)\vert^2 \ll H^{-2n}q^n \sum_{\mathbf{x}} \sum_{\substack{\mathbf{h_1},\mathbf{h_2} \le H:\\ \mathbf{x}+\mathbf{h_1}, \mathbf{x}+\mathbf{h_2} \le q}} e_q\left(a \left(\phi(\mathbf{x}+\mathbf{h_1})-\phi(\mathbf{x}+\mathbf{h_2})\right)\right).\]
Writing $\mathbf{y}=\mathbf{x}+\mathbf{h_2}$ and $\mathbf{h}=\mathbf{h_1}-\mathbf{h_2}$, this is equivalent to
\[\vert S(q,a)\vert^2 \ll H^{-2n}q^n \sum_{\vert\mathbf{h}\vert \le H} w(\mathbf{h})\sum_{\mathbf{y} \in \mathcal{R}(\mathbf{h})} e_q(a(\phi(\mathbf{y}+\mathbf{h})-\phi(\mathbf{y}))),\]
where $w(\mathbf{h})=\#\{\mathbf{h_1},\mathbf{h_2}: \mathbf{h}=\mathbf{h_1}-\mathbf{h_2}\} \le H^n$ and $\mathcal{R}(\mathbf{h})$ is a box as before. We have therefore shown that
\[\vert S(q,a)\vert^2 \ll H^{-n}q^n \sum_{\vert\mathbf{h}\vert \le H} \vert T(\mathbf{h},a,q)\vert,\]
where
\[T(\mathbf{h},a,q)=\sum_{\mathbf{y} \in \mathcal{R}(\mathbf{h})} e_q(a(\phi(\mathbf{y}+\mathbf{h})-\phi(\mathbf{y}))).\]
Again we reduce the degree of the form once more by squaring and expanding this expression to obtain
\[\vert T(\mathbf{h},a,q)\vert^2=\sum_{\mathbf{x},\mathbf{y} \in \mathcal{R}(\mathbf{h})} e_q(a(\phi(\mathbf{y}+\mathbf{h})-\phi(\mathbf{y})-\phi(\mathbf{x}+\mathbf{h})+\phi(\mathbf{x}))).\]
Writing $\mathbf{y}=\mathbf{x}+\mathbf{d}$ as before, this equals
\[\vert T(\mathbf{h},a,q)\vert^2=\sum_{\mathbf{d}} \sum_{\mathbf{x} \in \mathcal{S}(\mathbf{h},\mathbf{d})} e_q(a C(\mathbf{h},\mathbf{d},\mathbf{x}))\]
with $\mathcal{S}(\mathbf{h},\mathbf{d})$ the box and $C(\mathbf{x},\mathbf{y},\mathbf{z})$ the multilinear form defined before. Again we note that the inner sum is empty unless $\vert\mathbf{d}\vert \le 2q$. So we are in a situation very similar to the one in the previous section and the same argument developed there now shows that
\[\vert T(\mathbf{h},a,q)\vert^2 \ll q^{n+\varepsilon} N(a, q,\mathbf{h}),\]
where
\[N(a,q,\mathbf{h})=\#\left\{\mathbf{d} \in \mathbb{Z}^n: \vert\mathbf{d}\vert \le 2q, \left\|6 \frac{a}{q} B_i(\mathbf{h},\mathbf{d})\right\|<\frac{1}{q}\right\}.\]
Again applying Lemma \ref{geoz}, we find that
\[N(a,q,\mathbf{h}) \ll Z^{-n} \#\left\{\mathbf{d} \in \mathbb{Z}^n: \vert\mathbf{d}\vert \le 2qZ, \left\|6 \frac{a}{q} B_i(\mathbf{h},\mathbf{d})\right\|<\frac{Z}{q}\right\}\]
whenever $0<Z \le 1$. 

Note that of course the condition $\left\|6\frac{a}{q}B_i(\mathbf{h},\mathbf{d})\right\|<\frac{1}{q}$ already implies $q \mid 6B_i(\mathbf{h},\mathbf{d})$ but we have written it in this form so that we can recognize the condition to be of the same shape as in the earlier argument.

The next step is to apply Lemma \ref{lem.bootstrap} to turn the inequality into the equality $B_i(\mathbf{h},\mathbf{d})=0$. Here we choose $m=6B_i(\mathbf{h},\mathbf{d})$ so that $X \asymp MH qZ$ and $P_1 \asymp \frac{q}{Z}$. Thus any choice of $Z$ satisfying $MHZ  \ll 1$ for a sufficiently small implicit constant allows us to conclude that
\[N(a, q,\mathbf{h}) \ll Z^{-n} \#\left\{\mathbf{d} \in \mathbb{Z}^n: \vert \mathbf{d}\vert \le 2 qZ, B_i(\mathbf{h},\mathbf{d})=0\right\} \ll Z^{-n} ( qZ)^{n-r(\mathbf{h})}\]
and hence
\[\vert S(q,a)\vert^2 \ll \frac{q^{\frac{3n}{2}+\varepsilon}}{H^nZ^{\frac{n}{2}}} \sum_{\vert\mathbf{h}\vert \le H} ( qZ)^{\frac{n-r(\mathbf{h})}{2}}.\]
If we assume in addition that $C$ is $\psi$-good and $H \le M^{\psi}$, we obtain the bound
\begin{align*}
\vert S(q,a)\vert^2 &\ll \frac{q^{\frac{3n}{2}+\varepsilon}}{H^nZ^{\frac{n}{2}}}\sum_{r=0}^{14} H^{n-14+r+\varepsilon} ( qZ)^{\frac{n-r}{2}}\\
&\ll \frac{q^{2n+\varepsilon}}{H^{14}}\sum_{r=0}^{14} \frac{H^r}{(qZ)^{\frac{r}{2}}}\\
&\ll q^{2n+\varepsilon}\left(\frac{1}{H}+\frac{1}{qZ}\right)^{14}.
\end{align*}
Again we assumed that $qZ \ge 1$, but the final result is trivial otherwise. So again it will be optimal to choose $Z$ as large as possible i.e. $Z \asymp \frac{1}{MH}$. Inserting this into our result we obtain the bound
\begin{equation}\label{vdc1}\vert S(q,a)\vert \ll q^{n+\varepsilon}\left(\frac{1}{H^2}+\frac{HM}{q}\right)^{\frac{7}{2}}.\end{equation}
The final step is to choose the value of $H$ minimizing the RHS of (\ref{vdc1}). We are given the conditions $1 \le H \le q$ and $H \le M^{\psi}$. Putting $\gamma=\left(\frac{q}{M}\right)^{\frac{1}{3}}$, we see that the RHS of (\ref{vdc1}) is decreasing for $H \ll\gamma$ and increasing for $H \gg \gamma$ so that the optimal choice is $H \asymp \min(M^{\psi}, \gamma)$. Note that $H \le q$ is then automatically satisfied. We have thus proved the following result.

\begin{lemma}
\label{lem.S(q,a)bound}
Let $a$ and $q$ be coprime integers with $0 \le a<q$. Assume that $C$ is $\psi$-good. Then
\[\frac{S(q,a)}{q^n} \ll \left(\frac{M}{q}\right)^{\frac{7}{3}+\varepsilon}+M^{-7\psi+\varepsilon}.\]
\end{lemma}

\subsection{A mean square average via van der Corput}

In this section we finally apply the improved version of van der Corput's method developed in \cite{heath2007cubic} to obtain a satisfying bound for the minor arc contribution.

The idea is to exploit that the minor arc contribution already involves an average over both the modulus $q$ and the integration variable $\beta$.

From now on let our box be again $\mathscr{B}=\mathscr{B}(\mathbf{z})$ with the center $\mathbf{z}$ as chosen in Lemma \ref{lemma_center_B}.
Instead of a pointwise bound for $S(\alpha)$, we now seek to estimate the mean square average
\[M(\alpha,\kappa):=\int_{\alpha-\kappa}^{\alpha+\kappa} \vert S(\beta)\vert^2 d\beta,\]
where $\kappa \in (0,1)$ is a small parameter to be determined. By an appropriate dissection of the minor arcs and an application of Cauchy-Schwarz, a satisfactory  estimate for $M(\alpha,H)$ will allow us to bound the minor arc contribution $\int_{\mathfrak{m}} S(\alpha) d\alpha$.

We proceed as in the previous section, only now we consider the more general (but still trivial) identity
\[H_1H_2\dots H_nS(\beta)=\sum_{\mathbf{h}: h_i \le H_i} \sum_{\mathbf{x}+\mathbf{h} \in P\mathscr{B}} e(\beta \phi(\mathbf{x}+\mathbf{h}))=\sum_{\mathbf{x} \in \mathbb{Z}^n} \sum_{\mathbf{h}:\mathbf{x}+\mathbf{h} \in P\mathscr{B}} e(\beta \phi(\mathbf{x}+\mathbf{h}))\]
where $H_1,H_2,\dots,H_n \ge 1$ are certain parameters. We choose $H_1=P$ and $H_2=\dots=H_n=H$ for a certain parameter $H \le P$. Here the special rôle of the first variable comes from its special rôle in the construction of $\mathbf{z}$.

Note that the condition $H_i \le P$ ensures that the sum over $\mathbf{h}$ is non-empty only for $\mathcal{O}(P^n)$ values of $\mathbf{x}$. Squaring and applying Cauchy-Schwarz, we then find that
\[(H_1^2\dots H_n^2)\vert S(\beta)\vert^2 \ll P^n \sum_{\mathbf{x} \in \mathbb{Z}^n}  \left\vert \sum_{\mathbf{h}: \mathbf{x}+\mathbf{h}\in P\mathscr{B}} e(\beta \phi(\mathbf{x}+\mathbf{h}))\right\vert^2.\]
Opening the square, this yields
\[(H_1^2\dots H_n^2)\vert S(\beta)\vert^2 \ll P^n \sum_{\mathbf{x} \in \mathbb{Z}^n}  \sum_{\substack{\mathbf{h_1},\mathbf{h_2}:\\ \mathbf{x}+\mathbf{h_1}, \mathbf{x}+\mathbf{h_2} \in P\mathscr{B}}} e\left(\beta \left(\phi(\mathbf{x}+\mathbf{h_1})-\phi(\mathbf{x}+\mathbf{h_2})\right)\right).\]
Writing $\mathbf{y}=\mathbf{x}+\mathbf{h_2}$ and $\mathbf{h}=\mathbf{h_1}-\mathbf{h_2}$, this is equivalent to
\[(H_1^2\dots H_n^2)\vert S(\beta)\vert^2 \ll P^n \sum_{\vert h_i\vert \le H_i} w(\mathbf{h})\sum_{\mathbf{y} \in \mathcal{R}(\mathbf{h})} e(\beta(\phi(\mathbf{y}+\mathbf{h})-\phi(\mathbf{y})))\]
where $w(\mathbf{h})=\#\{\mathbf{h_1},\mathbf{h_2}: \mathbf{h}=\mathbf{h_1}-\mathbf{h_2}\} \le H_1H_2\dots H_n$. Instead of taking absolute values inside as before, we now first integrate over $\beta$. Here we use a smooth cutoff function to find that
\begin{align*}
M(\alpha,\kappa) &\le e \int_{\mathbb{R}} \exp\left(-\frac{(\beta-\alpha)^2}{\kappa^2}\right) \vert S(\beta)\vert^2 \mathrm{d}\beta\\
&\ll \frac{P^n}{(H_1\dots H_n)^2} \sum_{\mathbf{h}} w(\mathbf{h}) \sum_{\mathbf{y} \in \mathcal{R}(\mathbf{h})} I(\mathbf{h},\mathbf{y})
\end{align*}
and hence
\begin{equation}\label{estimate1}M(\alpha,\kappa) \ll \frac{P^n}{H_1\dots H_n} \sum_{\mathbf{h}} \left\vert \sum_{\mathbf{y} \in \mathcal{R}(\mathbf{h})} I(\mathbf{h},\mathbf{y})\right\vert,\end{equation}
where
\[ I(\mathbf{h},\mathbf{y})= \int_{\mathbb{R}} \exp\left(-\frac{(\beta-\alpha)^2}{\kappa^2}\right)e\left(\beta(\phi(\mathbf{y}+\mathbf{h})-\phi(\mathbf{y}))\right) \mathrm{d}\beta\]
which can also be written as
\begin{equation}\label{estimate2} I(\mathbf{h},\mathbf{y}) =\sqrt{\pi} \kappa \exp\left(-\pi^2\kappa^2\left(\phi(\mathbf{y}+\mathbf{h})-\phi(\mathbf{y})\right)^2\right) e(\alpha(\phi(\mathbf{y}+\mathbf{h})-\phi(\mathbf{y}))).\end{equation}
Our goal is to bound the contribution of the terms where $h_1$ is large so that we can effectively bound $h_1$ to a shorter interval. The point is that by our choice of the box $\mathscr{B}(\mathbf{z})$ we have a good lower bound for the partial derivative $\frac{\partial C}{\partial x_1}$ inside $\mathscr{B}$. But we expect $\phi(\mathbf{y}+\mathbf{h})-\phi(\mathbf{y}) \approx h_1 \frac{\partial \phi(\mathbf{y})}{\partial x_1} \approx h_1 \frac{\partial C(\mathbf{y})}{\partial x_1}$ so that this difference should be large if $h_1$ is large which means that $I(\mathbf{h},\mathbf{y})$ will be small.

Let us make this precise. By Lemma \ref{lemma_center_B} we have the bound
\begin{equation}\frac{\partial C}{\partial x_1} \gg M^6\end{equation}
on all of $\mathcal{B}$.

By homogeneity this implies that
 \[\frac{\partial C}{\partial x_1} \gg P^2M^{6}\]
on $P\mathscr{B}$. 

We thus find that
\[\frac{\partial \phi}{\partial x_1}=\frac{\partial C}{\partial X_1}+\mathcal{O}\left(PM\vert \mathbf{z}\vert\right) \gg P^2M^6\]
on all of $P\mathcal{B}$ as well.

Using $\vert\mathbf{y}\vert \ll PM^{3.75}$ for $\mathbf{y} \in P\mathscr{B}$, we now have the approximation
\[\phi(\mathbf{y}+\mathbf{h})-\phi(\mathbf{y})=h_1 \cdot \frac{\partial \phi}{\partial x_1}(\mathbf{y})+\mathcal{O}\left(HP^2 \max_i \left\vert \frac{\partial \phi}{\partial x_i}\right\vert+h_1^2PM^{4.75}\right).\]
Note that $h_1 \le P$ implies that
\[h_1^2PM^{4.75} \ll h_1P^2M^{4.75}\]
and so we will have
\[\phi(\mathbf{y}+\mathbf{h})-\phi(\mathbf{y}) \gg \vert h_1\vert  \cdot P^2M^6\]
unless $\vert h_1\vert \ll H$. Indeed, unless also $\vert h_1\vert \ll \frac{(\log P)^2}{\kappa P^2M^6}$, this means that
\[\vert C(\mathbf{y}+\mathbf{h})-C(\mathbf{y})\vert \ge \frac{(\log P)^2}{\kappa}\]
and it is easy to see from \eqref{estimate1} and \eqref{estimate2} that the contribution to $M(\alpha,\kappa)$ of such $\mathbf{h}$ is $\mathcal{O}(1)$.
Here we have made use of the earlier assumption that $\frac{\partial C}{\partial x_i}$ and hence $\frac{\partial \phi}{\partial x_1}$ is maximal for $i=1$.

Hence we have shown that
\[M(\alpha,\kappa) \ll 1+\frac{P^{n-1}}{H^{n-1}} \sum_{\vert h_i\vert \ll H} \left\vert \sum_{\mathbf{y}} I(\mathbf{h},\mathbf{y})\right\vert\]
if we choose $\kappa \asymp \frac{(\log P)^2}{HP^2M^6}$.

Moreover, the range $\vert \beta-\alpha\vert \ge \kappa \log P$ in the definition of $I(\mathbf{h},\mathbf{y})$ clearly has a total contribution of $\mathcal{O}(1)$ to $M(\alpha,\kappa)$ so that we end up with the estimate
\[M(\alpha,\kappa) \ll 1+ \frac{P^{n-1}}{H^{n-1}} \sum_{\vert h_i\vert \ll H} \int_{\alpha-\kappa \log P}^{\alpha+\kappa \log P} \vert T(\mathbf{h},\beta)\vert d\beta,\]
where
\[T(\mathbf{h},\beta)=\sum_{\mathbf{y} \in \mathcal{R}(\mathbf{h})} e\left(\beta (\phi(\mathbf{y}+\mathbf{h})-\phi(\mathbf{y}))\right).\]

The same argument employed already several times now yields
\[\vert T(\mathbf{h},\beta)\vert^2 \ll P^{n+\varepsilon} N(\beta,P,\mathbf{h})\]
where
\[N(\beta,P,\mathbf{h})=\#\left\{\mathbf{d} \in \mathbb{Z}^n: \vert \mathbf{d}\vert \le 2 P, \left\|6\beta B_i(\mathbf{h},\mathbf{d})\right\|<\frac{1}{2P}\right\},\]
so that
\begin{equation} M(\alpha,\kappa) \ll 1+\frac{\kappa P^{\frac{3n}{2}-1+\varepsilon}}{H^{n-1}} \sum_{\vert h_i\vert \ll H} \max_{\beta \in I} N(\beta,P,\mathbf{h})^{\frac{1}{2}}\end{equation}
where $I=\{\beta: \vert \beta-\alpha\vert \le \kappa \log P\}$.

We next claim that 
\[\max_{\beta \in I} N(\beta,P,\mathbf{h}) \ll P^{\varepsilon} N(\alpha,P,\mathbf{h}).\]
Indeed, consider a vector $\mathbf{d}$ counted by $N(\beta,P,\mathbf{h})$. By our assumption, it satisfies $\vert\mathbf{d}\vert \ll P$ as well as $\|6\beta B_i(\mathbf{h},\mathbf{d})\| \ll \frac{1}{P}$ so that
\[\|6\alpha B_i(\mathbf{h},\mathbf{d})\| \ll \frac{1}{P}+\vert \beta-\alpha\vert \vert B_i(\mathbf{h},\mathbf{d})\vert \ll \frac{1}{P}+\kappa(\log P)MHP \ll \frac{1}{P}\]
by our choice of $\kappa$.

We conclude that
\[\max_{\beta} N(\beta,P,\mathbf{h}) \le \#\left\{\mathbf{d} \in \mathbb{Z}^n: \vert\mathbf{d}\vert \ll P, \left\|6\alpha B_i(\mathbf{h},\mathbf{d})\right\| \ll \frac{1}{P}\right\} \ll N(\alpha,P,\mathbf{h}),\]
where the last estimate is a consequence of Lemma \ref{lem.shrinkinglemma} with $Z \asymp 1$ sufficiently small.

We conclude that

\begin{equation}\label{Malphakappa} M(\alpha,\kappa) \ll 1+\frac{\kappa P^{\frac{3n}{2}-1+\varepsilon}}{H^{n-1}} \sum_{\vert h_i\vert \ll H} N(\alpha,P,\mathbf{h})^{\frac{1}{2}}.\end{equation}

We now write $\alpha=\frac{a}{q}+\theta$ in a preparation for applying Lemmas \ref{lem.shrinkinglemma} and \ref{lem.bootstrap}. Indeed, Lemma \ref{lem.shrinkinglemma} implies that
\[N(\alpha,P,\mathbf{h}) \ll Z^{-n} \#\{\mathbf{d} \in \mathbb{Z}^n: \vert \mathbf{d}\vert<ZP, \left\|6\alpha B_i(\mathbf{h},\mathbf{d})\right\| \ll \frac{Z}{P}\}.\]
Following Heath-Brown, we apply this with two different choices of $Z$. In the first application we choose $Z=Z_1$ sufficiently small to ensure $B_i(\mathbf{h},\mathbf{d})=0$ as before. In the second application however, we make a larger choice of $Z=Z_2$ which only forces that $q \mid B_i(\mathbf{h},\mathbf{d})$. We then have to consider the implications of this weaker result. It was observed in \cite{heath2007cubic} that only this new trick allows us to handle the case of $14$ variables. 

To apply Lemma \ref{lem.bootstrap} with $m=6B_i(\mathbf{h},\mathbf{d})$ we have to choose $X \asymp MH PZ$ and $P_1 \asymp \frac{P}{Z}$ so that in our application of Lemma \ref{geoz} we need to choose $Z \le 1$ satisfying
\[\vert \theta\vert \ll \frac{1}{MH PZq} \quad \text{ and } \quad Z \ll \frac{P}{q}\]
with sufficiently small implicit constants.
In the first application we should also have
\[ MH PZ \ll q \quad \text{ or } \quad Z \ll q\vert\theta\vert P.\]
Writing 
\begin{equation}\label{etadef}\eta=\vert \theta\vert+\frac{1}{P^2HM}\end{equation}
for convenience and assuming $q \sim R$, this means that we need to choose
\[Z_1 \asymp \min\left(R\eta P,\frac{1}{RHMP\eta}\right),\]
noting that this automatically implies $Z_1 \le 1$. Similarly, we should choose
\[Z_2 \asymp \min\left(1,\frac{1}{RHMP\eta}\right).\]

In the application with $Z=Z_1$, we thus find that
\begin{align*}
N(\alpha,P,\mathbf{h})&\ll Z_1^{-n} \#\{\mathbf{w} \in \mathbb{Z}^n: \vert\mathbf{w}\vert \ll P, B_i(\mathbf{h},\mathbf{d})=0\}\\
&\ll Z_1^{-n} (Z_1 P)^{n-r}\\
&\ll P^n\left((R\eta P^2)^{-r}+(RHM\eta)^r\right)
\end{align*}
where $r=r(\mathbf{h})$. The argument only works for $Z_1P \ge 1$ but the estimate is true in any case because of the trivial bound $N(\beta,P,\mathbf{h}) \ll P^n$.

On the other hand, in the application with $Z=Z_2$, we obtain
\[N(\alpha,P,\mathbf{h}) \ll Z_2^{-n} \#\{\mathbf{d} \in \mathbb{Z}^n: \vert\mathbf{d}\vert \ll Z_2P, q \mid B_i(\mathbf{h},\mathbf{d})\}.\]
To make this a useful estimate, we need to count vectors $\mathbf{d}$ with $q \mid B_i(\mathbf{h},\mathbf{d})$. Here we can copy the results from \cite{heath2007cubic}, but we need to introduce some notation. Recall that we have fixed a vector $\mathbf{h}$ with $r(\mathbf{h})=r$, meaning that the matrix $M(\mathbf{h})$ has rank $r$. We now distinguish primes $p$ according to whether $p$ divides all the $r \times r$ minors of $M(\mathbf{h})$. If it does, we say that $p$ is of type I and if it does not, we say that it is of type II. We then decompose $q=q_1q_2$ such that $q_1$ is a product of type I primes and $q_2$ a product of type II primes. The argument in \cite[p.\,218\,f.]{heath2007cubic} now shows that
\[\#\{\mathbf{d} \in \mathbb{Z}^n: \vert\mathbf{d}\vert \ll Z_2P, q \mid B_i(\mathbf{h},\mathbf{d})\} \ll \left(1+\frac{B}{q_2}\right)^rB^{n-r}\]
with $B=1+Z_2 P$. If $Z_2 P \gg 1$ so that $B \asymp Z_2 P$, this shows that
\begin{align*}
N(\alpha,P,\mathbf{h}) &\ll Z_2^{-n}\left(1+\frac{Z_2 P}{q_2}\right)^r(Z_2 P)^{n-r}\\
&=P^n\left(\frac{1}{q_2}+\frac{1}{Z_2 P}\right)^r\\
&\ll P^n\left(\frac{1}{q_2^r}+\frac{1}{P^r}+(RHM\eta)^r\right),
\end{align*}
but the intermediate and hence the final result is trivially true if $Z_2P \ll 1$.
Combining the results of the two applications, we obtain that
\[N(\alpha,P,\mathbf{h}) \ll P^n\left(\frac{1}{P^r}+(RHM\eta)^r+\min\left(\frac{1}{(R\eta P^2)^r},\frac{1}{q_2^r}\right)\right).\]
We now need to insert this into our estimate for $M(\alpha,\kappa)$ but we also want to introduce an average over $q$ to make use of the fact that $q_2$ is almost as large as $q$ most of the time.

Our object of study thus becomes
\begin{equation}
    \label{eq.defA}
    A(\theta,R,H,P):=\sum_{q \sim R} \sum_{(a;q)=1} \sum_{\vert h_i\vert \ll H}  N(\alpha,P,\mathbf{h})^{\frac{1}{2}},
\end{equation}
where we continue to write $\alpha=\frac{a}{q}+\theta$ and we remind the reader of our notation $q \sim R$ for the dyadic condition $R<q \le 2R$.

The argument above now leads to the bound
\[A(\theta,R,H,P) \ll R P^{\frac{n}{2}} \sum_{\vert h_i\vert \ll H} \sum_{q \sim R} \left[\frac{1}{P}+RHM\eta+\min\left(\frac{1}{R\eta P^2},\frac{1}{q_2}\right)\right]^{\frac{r(\mathbf{h})}{2}}.\]
We then need to estimate
\[V(\mathbf{h},R,\eta):=\sum_{q \sim R} \min\left(\frac{1}{R\eta P^2},\frac{1}{q_2}\right)^{r/2}\]
for $r=r(\mathbf{h})$. A double dyadic decomposition leads to
\begin{align*}
    V(\mathbf{h},R,\eta) &\ll P^{\varepsilon} \max_{S \le R} \sum_{q_1 \sim S} \sum_{q_2 \sim \frac{R}{S}} \min\left(\frac{1}{R\eta P^2}, \frac{S}{R}\right)^{r/2}\\
    &\ll P^{\varepsilon} \max_{S \le R} \frac{R}{S} \min\left(\frac{1}{R\eta P^2}, \frac{S}{R}\right)^{r/2} \#\{q_1 \le 2S\}.
\end{align*}
Now recall that $q_1$ only contains prime factors dividing a certain non-zero $r \times r$ determinant $M_0$ of $M(\mathbf{h})$. In particular, $M_0 \ll M^rH^r$. Applying Rankin's trick it now follows that
\[\#\{q_1 \le 2S\} \ll S^{\varepsilon} \sum_{q_1} q_1^{-\varepsilon}=S^{\varepsilon} \prod_{p \mid M_0} \frac{1}{1-p^{-\varepsilon}} \ll S^{\varepsilon} M_0^{\varepsilon} \ll M^{\varepsilon}\]
and hence
\[V(\mathbf{h},R,\eta) \ll M^{\varepsilon}\frac{R}{S} \min\left(\frac{1}{(R\eta P^2)^{\frac{r}{2}}},\left(\frac{S}{R}\right)^{\frac{r}{2}}\right).\]
Maximizing this function for $S$ we find that
\[V(\mathbf{h},R,\eta) \ll M^{\varepsilon} \frac{R}{(R\eta P^2)^{\frac{r}{2}}} \cdot \min(1,\eta P^2)^{e(r)},\]
where $e(0)=0, e(1)=\frac{1}{2}$ and $e(r)=1$ for $r \ge 2$. Assuming that $C$ is $\psi$-good and $H \le M^{\psi}$ it now follows that
\begin{align*}&A(\theta,R,H,P) \ll R^2 P^{\frac{n}{2}} \sum_{\vert h_i\vert \ll H} \left[\frac{1}{P^{\frac{r(\mathbf{h})}{2}}}+(RHM\eta)^{\frac{r(\mathbf{h})}{2}}+R^{-1}V(\mathbf{h},R,\eta)\right]\\
&\ll R^2 P^{\frac{n}{2}}M^{\varepsilon}  \sum_{\vert h_i\vert \ll H} \left[\frac{1}{P^{\frac{r(\mathbf{h})}{2}}}+(RHM\eta)^{\frac{r(\mathbf{h})}{2}}+\frac{1}{(R\eta P^2)^{\frac{r(\mathbf{h})}{2}}} \cdot \min(1,\eta P^2)^{e(r(\mathbf{h}))}\right]\\
&\ll R^2 P^{\frac{n}{2}}M^{\varepsilon} \sum_{r=0}^{14} H^{n-14+r}\left[\frac{1}{P^{\frac{r}{2}}}+(RHM\eta)^{\frac{r}{2}}+\frac{1}{(R\eta P^2)^{\frac{r}{2}}} \cdot \min(1,\eta P^2)^{e(r)})\right]\\
&\ll R^2 P^{\frac{n}{2}}H^nM^{\varepsilon} \left(\frac{1}{H^{14}}+\frac{1}{P^7}+(RHM\eta)^7+\frac{1}{(R\eta P^2)^7} \cdot \min(1,\eta P^2)\right).
\end{align*}
Finally, let us show that the term $\frac{1}{P^7}$ is negligble. Indeed, if $HRMP\eta \ge 1$, it is dominated by the third summand.. Otherwise, if $HRMP\eta \le 1$, we have $(R\eta P)^7 \le R\eta P \le \frac{1}{HM} \le \min(1,\eta P^2)$ on recalling that $\eta \ge \frac{1}{P^2HM}$ and hence $\frac{1}{P^7}$ is dominated by the last summand in that case.
In any case, we now conclude that
\begin{equation}
\label{eq.boundA}
A(\theta,R,H,P) \ll R^2 P^{\frac{n}{2}}H^nM^{\varepsilon} \left(\frac{1}{H^{14}}+(RHM\eta)^7+\frac{1}{(R\eta P^2)^7} \cdot \min(1,\eta P^2)\right).
\end{equation}

\section{Intermezzo: The singular series}

We are now ready to use the bounds for $S(q,a)$ in order to bound the singular series $\mathfrak{S}(P_0)$. However, we first require another ingredient which is a lower bound on the individual $p$-adic densities. The non-vanishing of these is a consequence of the existence of a non-singular $p$-adic solution and Hensel's Lemma. To get a uniform lower bound, we require a quantitative result on how singular such a solution is.

This is typically described in terms of the invariant $\Delta(C)$ which is defined as the greatest common factor of all $n \times n$-minors of the $n \times \binom{n+1}{2}$-matrix $(c_{ijk})_{i,(j,k)}$. In particular, $\Delta(C) \ne 0$ whenever $C$ is non-degenerate. We also note that if $C$ is degenerate modulo $q$, then $q \mid \Delta(C)$.

We note that sometimes the invariant we called $\Delta(C)$ is also denoted as $h(C)$, but this already has a different meaning in our work.

We also define $\Delta(\phi)$ to be $\Delta(\widetilde{\phi})$ where $\widetilde{\phi}$ is the homogenized version of $\phi$, i.e. a cubic form in $n+1$ variables. Note that $\Delta(C) \mid \Delta(\phi)$.

We will throughout work with the following consequence of Hensel's Lemma:

\begin{lemma}
If $\phi(\mathbf{x}) \equiv 0 \pmod{p^{2\ell-1}}$ and $p^{\ell} \nmid \nabla \phi(\mathbf{x})$, then $x$ lifts to a non-singular $p$-adic solution.
\end{lemma}

\medskip

In the homogeneous case, Davenport \cite[Lemma 18.7]{davenport_book} established the following:
\begin{lemma}
\label{lem.hom_non_singular_padic}
If $C \in \mathbb{Z}[x_1,\dots,x_n]$ is a non-degenerate cubic form in $n \ge 10$ variables, then for each prime $p$ there is a $p$-adic solution $\mathbf{x}$ to $C(\mathbf{x})=0$ such that $p^{\ell} \nmid \nabla C(\mathbf{x})$ for
\[\ell=\ell_C(p):= 3 \cdot \left\lfloor \frac{v_p(\Delta(C))}{n-9}\right\rfloor+3.\]
\end{lemma}

In the inhomogeneous case, as observed by Davenport and Lewis~\cite{dl64}, the Necessary Congruence Condition is in general not enough to deduce the existence of a non-singular $p$-adic solution to $\phi(\mathbf{x})=0$. A counterexample in the $14$ variables $x_i$ for $1 \le i \le 4$ and $x_{i,j}$ for $1 \le i \le j \le 4$ is given by
\begin{align*}
    &\phi(\mathbf{x})=x_1^2-Nx_2^2+p(x_3^2-Nx_4^2)+p^2\left(\sum_{1 \le i \le j \le 4} x_{i,j} x_ix_j\right)
\end{align*}
for a quadratic non-residue $N$ modulo $p$. However, they managed to prove that the desired implication holds true if $n \ge 15$ and $\phi$ is non-degenerate. A quantitative version of their argument leads to the following result:

\begin{lemma}
\label{lem.inhom_nonsing_padic}
If $\phi \in \mathbb{Z}[x_1,\dots,x_n]$ is a non-degenerate cubic polynomial in $n \ge 15$ variables, then for each prime $p$ there is a $p$-adic solution $\mathbf{x}$ to $\phi(\mathbf{x})=0$ such that $p^{\ell} \nmid \nabla \phi(\mathbf{x})$ for
\[\ell = \ell_{\phi}(p):=\begin{cases} 98 & p \nmid \Delta(\phi)\\ 144v_p(\Delta(\phi))+2 & p \mid \Delta(\phi) \end{cases}.\]
\end{lemma}
While the result is probably not optimal and the proof has a certain ad-hoc structure, we note that the counterexample for $14$ variables naturally arises from the structure of the proof. We will momentarily see that one could prove a superficially stronger bound, but in the critical case $v_p(\Delta(\phi))=1$ we do not lose anything.

\begin{proof}
Let $k=v_p(\Delta(\phi))$ and $m=\max(6\left\lfloor \frac{k}{n-9}\right\rfloor+5, k+1)$. We will prove the result for $\ell=48k+16m+18$. Note that for $k=0$, we have $m=5$ so that $\ell=98$. For $k \ge 1$, we have $m \le 6k-1$ and hence $\ell \le 144k+2$ as desired. Here, we used that $\left\lfloor \frac{k}{n-9}\right\rfloor \le k-1$ for $k \ge 1$ and $n \ge 11$.

\medskip

First of all, using the congruence condition, we can find a solution modulo $p^{96k+32m+35}$ and after translating the variables appropriately, we may assume that this solution is given by $(0,0,\dots,0)$ so that
\[\phi(\mathbf{x}) \equiv C(\mathbf{x})+Q(\mathbf{x})+L(\mathbf{x}) \pmod {p^{96k+32m+35}}.\]
If $L$ does not vanish identically modulo $p^{48k+16m+18}$, this solution $\mathbf{x}=(0,0,\dots,0)$ satisfies $p^{48k+16m+18} \nmid \nabla \phi(\mathbf{x})$ and $p^{2(48k+16m+18)-1} \mid \phi(\mathbf{x})$ so that it lifts to a $p$-adic solution of the desired shape by Hensel's Lemma.

Otherwise, we may assume that 
\[\phi(\mathbf{x}) \equiv C(\mathbf{x})+Q(\mathbf{x}) \pmod {p^{48k+16m+18}}.\]
If $Q$ vanishes identically modulo $p^{m}$, then $C$ is non-degenerate modulo $p^{m}$ since otherwise $p^{k+1} \mid \Delta(C)$ by construction of $m$, contradicting the definition of $k$.

But if $C$ is non-degenerate modulo $p^m$, by Lemma \ref{lem.hom_non_singular_padic}, there is a solution of $C(\mathbf{x})=0$ with $v_p(\nabla C(\mathbf{x})) \le 3\left\lfloor \frac{k}{n-9}\right\rfloor+3$ and since $m \ge 6\left\lfloor \frac{k}{n-9}\right\rfloor+5$, this lifts to a non-singular solution by Hensel's Lemma.

From now on we assume that $Q$ does not vanish identically modulo $p^m$.

\textbf{First case:} $Q$ has rank at least five modulo $p^{M}$ where $M=12k+4m+5$. Then we can find a non-singular solution $\beta$ of $Q(\mathbf{\beta})=0$ with $p^{M} \nmid \nabla Q(\mathbf{\mathbf{\beta}})$. Rescaling by $p^{2M-1}$, we have $p^{6M-3} \mid \phi(p^{2M-1}\mathbf{\beta})$ and $p^{3M-1} \nmid \nabla \phi(p^{2m-1}\mathbf{\beta})$ and we obtain a non-singular solution with $p^{3M-1}=p^{36k+12m+15} \nmid \nabla \phi(\mathbf{x})$. Note that here we used that $4M-2=48k+16m+18$ so that $p^{6M-3} \mid L(p^{2M-1}\beta)$.

\textbf{Second case:} $Q$ has rank $1 \le r \le 4$ modulo $p^M$. Hence $\phi(\mathbf{x})$ is equivalent to a form $\psi(\mathbf{y})$ with
\[\psi(\mathbf{y}) \equiv y_1R_1(\mathbf{y})+\dots+y_rR_r(\mathbf{y})+R(y_1,\dots,y_r)+\Gamma(y_{r+1},\dots,y_n) \pmod{p^{M}}\]
for some cubic form $\Gamma$ and some quadratic forms $R,R_1,\dots,R_r$.

\textbf{First subcase:} $\Gamma$ does not vanish identically modulo $p^{2k+1}$. We then choose $\delta_1,\dots,\delta_r$ such that $p^{m} \nmid R(\delta_1,\dots,\delta_r)$ and $\delta_{r+1},\dots,\delta_n$ such that $p^{2k+1} \nmid \Gamma(\delta_{r+1},\dots,\delta_n)$. Let $\rho \le 2k$ be chosen so that $p^{\rho} \| \Gamma(\delta_{r+1},\dots,\delta_n)$.

We then choose $\varepsilon=(p^{\rho+1}\delta_1,\dots,p^{\rho+1}\delta_r,\delta_{r+1},\dots,\delta_n)$ so that $p^{\rho} \| y_1R_1+\dots+y_rR_r+\Gamma$ if we insert $\varepsilon$.

We can thus find a $p$-adic integer $\mu$ such that
\[\mu(y_1R_1+\dots+y_rR_r+\Gamma)+R=0,\]
still everything evaluated at $\varepsilon$. Choosing $y=\mu \varepsilon$, we then have $p^M \mid \psi(\mathbf{y})$.

Moreover, by Euler's identity we have that
\[\sum_j \varepsilon_j \psi^{(j)}(\mu \varepsilon)\equiv -\mu \cdot R(\varepsilon) \pmod{p^M}\]
and since 
\[v_p(\mu \cdot R(\varepsilon)) \le 2v_p(R(\varepsilon))-\rho \le 3\rho+4+2v_p(R(\delta_1,\dots,\delta_r))\]
\[\le 2m+3\rho+2 \le 2m+6k+2,\]
we have that one of the $\psi^{(j)}$ is divisible at most by $2m+6k+2$ so that we have a solution with $\ell=2m+6k+3$ modulo $p^M$ lifting by Hensel since $M= 2\ell-1$.

\textbf{Second subcase:} $\Gamma$ vanishes identically modulo $p^{2k+1}$. Then we can write
\[\psi(\mathbf{y}) \equiv y_1R_1(\mathbf{y})+\dots+y_rR_r(\mathbf{y})+R(y_1,\dots,y_r) \pmod{p^{2k+1}}.\]
Then we find a solution of the shape $(0,\dots,0,y_{r+1},\dots,y_n)$ with $p^{k+1} \nmid \nabla \psi$ (which then again lifts by Hensel) unless all the $R_i$ vanish modulo $p^{k+1}$ at these points. But this means that the variables $y_{r+1},\dots,y_n$ appear at most linearly in all of the $R_i$ so that we can write
\[\psi(\mathbf{y}) \equiv y_{r+1}S_{r+1}(y_1,\dots,y_r)+\dots+y_nS_n(y_1,\dots,y_r)+R(y_1,\dots,y_r) \pmod{p^{k+1}}.\]
But then, finally, as there are only at most $\binom{r+1}{2} \le 10$ linearly independent quadratic monomials in $y_1,\dots,y_r$ and $n-r \ge 11$, the $S_i$ can not be linearly independent and so $\psi$ and hence $\phi$ must be degenerate modulo $p^{k+1}$ so that $p^{k+1} \mid \Delta(\phi)$ contradicting the definition of $k$.
\end{proof}

We are now equipped with everything needed for a lower bound of the singular series.

To this end, for each prime $p$ let
\[k_C(p)=\begin{cases} \max_{t \in \mathbb{N}: p^t \le P_0} \{t\}, & p \nmid \Delta(C),\\\max_{t \in \mathbb{N}: p^t \le P_0} \{t,2\ell_C(p)-1\}, & p \mid \Delta(C) \end{cases}\]
in the homogeneous case and
\[k_{\phi}(p)=\begin{cases} \max_{t \in \mathbb{N}: p^t \le P_0} \{t\}, & p \nmid \Delta(\phi),\\\max_{t \in \mathbb{N}: p^t \le P_0} \{t,2\ell_{\phi}(p)-1\}, & p \mid \Delta(\phi) \end{cases}\]
in the inhomogeneous case. We then define the truncated Euler product
\[S_{\phi}(P_0)=\prod_{p \le P_0} \sum_{i=0}^{k_{\phi}(p)} A(p^i),\]
where
\[A(q)=\sum_{(a;q)=1} \frac{S(q,a)}{q^n}\]
and we recall the classical fact that
\[\sum_{i=0}^k A(p^i)=p^{-k(n-1)} \rho(p^k),\]
where $\rho(p^k)$ denotes the number of solutions of $\phi(\mathbf{x}) \equiv 0 \pmod{p^k}$.

Similarly, we define
\[S_C(P_0)=\prod_{p \le P_0} \sum_{i=0}^{k_C(p)} A(p^i).\]
We will first establish a lower bound for the truncated Euler product $S(P_0)$ and then estimate the difference to the truncated singular series $\mathfrak{S}(P_0)$.

We first deal with the primes not dividing $\Delta$. The key ingredient here is the following bound which is Lemma 9 in \cite{bde12}.

\begin{lemma}
\label{lem.beitrag_pnondeg}
Let $C$ be a cubic form in $n \ge 10$ variables. Then for any $p \gg 1$ with $p \nmid \Delta(C)$ and any $k \ge 1$, we have
\[\rho^*(p^k) \ge p^{k(n-1)} \left(1+\mathcal{O}\left(\frac{1}{p}\right)\right),\]
where $\rho^*(p^k)$ denotes the number of non-singular solutions of $\phi(\mathbf{x}) \equiv 0 \pmod{p^k}$.
\end{lemma}

In the homogeneous case, this immediately shows that the contribution from the primes not dividing $\Delta(C)$ to $S_C(P_0)$ is $\gg P_0^{-\varepsilon}$.

In the inhomogeneous case, we note that
\[\rho_{\phi}^*(p)=\frac{\rho_{\widetilde{\phi}}^*(p)-\rho_C^*(p)}{p-1} \ge \frac{p^n+\mathcal{O}(p^{n-1})}{p-1} \ge p^{(n-1)} \left(1+\mathcal{O}\left(\frac{1}{p}\right)\right)\]
whenever $p \nmid \Delta(\phi)$ by applying Lemma \ref{lem.beitrag_pnondeg} to the cubic form $\widetilde{\phi}$. Here, we used that $\rho_C^*(p) \ll p^{n-1}$ for any cubic form $C$. This follows from $\rho_C^*(p) \le \rho_C(p) \ll p^{n-1}$ unless $C$ vanishes modulo $p$, but in that latter case we have $\rho_C^*(p)=0$.

We thus obtain the same bound as in Lemma \ref{lem.beitrag_pnondeg} for $\rho^*(p^k)$ by Hensel's Lemma and thus deduce that the contribution from the primes not dividing $\Delta(\phi)$ to $S_{\phi}(P_0)$ is also $\gg P_0^{-\varepsilon}$.

Note that in both cases the contribution from the primes $p \ll 1$ not dividing $\Delta$ is clearly $\gg 1$.

\medskip

We now need to deal with the primes $p \mid \Delta$.

In the inhomogeneous case, we conclude from Lemma \ref{lem.inhom_nonsing_padic} that
\[\rho(p^{k(p)}) \gg p^{k(p)(n-1)-\ell(p)(n-1)},\]
so that the contribution to $S(P_0)$ from the primes dividing $\Delta(\phi)$ is
\[\gg M^{-\varepsilon}\prod_{p \mid \Delta(\phi)} p^{-(n-1)\ell_{\phi}(p)} \gg M^{-\varepsilon}\prod_{p \mid \Delta(\phi)} p^{-292(n-1)v_p(\Delta(\phi))}\gg M^{-292(n^2-1)-\varepsilon},\]
using that $\Delta(\phi) \ll M^{n+1}$.

In the homogeneous case, we can do a bit better using the ideas from \cite{bde12}. We first consider the case that modulo $p$ the form $C$ is not equivalent to one in less than four variables.

In that case, by Lemma 10 from \cite{bde12}, we have
\[\rho_C^*(p^k) \ge p^{k(n-1)} \left(1+\mathcal{O}\left(p^{-1/2}\right)\right).\]

We thus conclude that primes $p \mid \Delta$ such that the order of $\widetilde{\phi}$ resp. $C$ modulo $p$ is at least four, have
\[\rho(p^{k(p)}) \gg p^{k(p)(n-1)},\]
at least for $p \gg 1$.

Finally, we need to deal with the case where the order modulo $p$ is $t \le 3$.

Hence, $C$ is equivalent to a cubic form $C_1(x_1,\dots,x_t)$ modulo $p$. If $C_1$ has a non-singular zero, then immediately $\rho_C^*(p) \ge p^{n-t} \ge p^{n-3}$ as we can vary $x_{t+1},\dots,x_n$ arbitrarily. 
Otherwise, if w.l.o.g. $(1,0,\dots,0)$ is a singular zero modulo $p$, then
\[C_1(x_1,\dots,x_t) \equiv x_1Q(x_2,\dots,x_t)+C_2(x_2,\dots,x_t) \pmod{p}.\]
Here, $Q$ cannot vanish identically modulo $p$ as otherwise the order would be at most $t-1$. We can thus choose $x_2,\dots,x_t$ such that $p \nmid Q$ and then solve the congruence for $x_1$ to obtain a non-singular solution and conclude as above.

Finally, we need to consider the case where $C_1$ does not have a non-trivial zero modulo $p$, i.e all the roots have all variables divisible by $p$.

But then if we define
\[C'(X_1,\dots,X_n)=p^{-1}C(pX_1,\dots,pX_t,X_{t+1},\dots,X_n),\]
we have $\rho_C(p^k)=p^{n-t} \rho_{C'}(p^{k-1})$. We can iterate this argument and after at most $\ell_C(p)-1$ steps we will end up with a cubic form that has a non-singular solution modulo $p$ which we can then treat as above.

Since we lose a factor of at most $p^2$ in each step ($p^{n-3}$ instead of $p^{n-1}$), we eventually end up with the bound
\[\rho_C(p^{k(p)}) \gg p^{k(p)(n-1)-2\ell(p)}\]
and therefore the total contribution to $S(P_0)$ of these primes is
\[\gg M^{-\varepsilon}\prod_{p \mid \Delta(C)} p^{-2\ell(p)} \gg M^{-\varepsilon} \prod_{p \mid \Delta(C)} p^{-6v_p(\Delta(C))} \gg M^{-\varepsilon} \Delta(C)^{-6} \gg M^{-6n-\varepsilon},\]
using $\Delta(C) \ll M^n$.

Here we again used that $\ell(p)=3\left\lfloor \frac{v_p(\Delta(C))}{n-9}\right\rfloor+3 \le 3v_p(\Delta(C))$ for $p \mid \Delta(C)$ which is sharp in the critical case $v_p(\Delta(C))=1$. The argument in \cite{bde12} is not correct and fails in precisely that critical case.

We summarize our results as follows:

\begin{lemma}
We have
\[S_C(P_0) \gg M^{-6n-\varepsilon}\]
and
\[S_{\phi}(P_0) \gg M^{-292(n^2-1)-\varepsilon}.\]
\end{lemma}

Finally, we need to estimate the difference between the truncated singular series $\mathfrak{S}(P_0)$ and $S(P_0)$. To this end, define
\[\mathscr{Q}(P_0)=\{q \in \mathbb{N}: q>P_0, p^i \mid q \Rightarrow p \le P_0 \text{ and } i \le k(p)\}\]
where $k(p)$ is either $k_C(p)$ or $k_{\phi}(p)$, depending on the context.

It is then clear that we have
\[R(P_0):=\vert \mathfrak{S}(P_0)-S(P_0)\vert \le \sum_{q \in \mathscr{Q}(P_0)} \vert A(q)\vert.\]
For the case of non-singular forms and the inhomogeneous case, we then have the following result:

\begin{lemma}
If $C$ is $\infty$-good, then 
\[R(P_0) \ll M^{\frac{7}{3}} P_0^{-\frac{1}{3}+\varepsilon}.\]
\end{lemma}
\begin{proof}
From Lemma \ref{lem.S(q,a)bound} we have $A(q) \ll M^{\frac{7}{3}} q^{-\frac{4}{3}+\varepsilon}$ and the claim follows immediately by summing over $q>P_0$.
\end{proof}

In the general case of a cubic form in $14$ variables, we need to work slightly harder. The result is the following.

\begin{lemma}
If $C$ is $\psi$-good and $p^{k(p)} \le M^{1+3\psi}$ for all $p \le P_0$ and $\delta>2$ satisfies
\begin{equation}
    \tag{$\mathfrak{S}_1$}
    0<\frac{14}{14-6\delta}<1+3\psi,
\end{equation}
then
\[R(P_0) \ll M^{\frac{14\delta}{14-6\delta}} P_0^{2-\delta+\varepsilon}.\]
\end{lemma}
\begin{proof}
This is similar to the proof of Lemma 13 in \cite{bde12} which seems to miss the factor $\delta$ in the exponent of $M$.

We first note that the bound
\[A(q) \ll M^{\frac{7}{3}} q^{-\frac{4}{3}+\varepsilon}+q^{1+\varepsilon}M^{-7\psi+\varepsilon}\]
from Lemma \ref{lem.S(q,a)bound} implies that 
\[A(q) \ll M^{\frac{7}{3}} q^{-\frac{4}{3}+\varepsilon} \ll q^{1-\delta}\] holds uniformly for
\begin{equation}
\label{eq.rangeforq}
M^{\frac{14}{14-6\delta}+\varepsilon} \ll q \ll M^{1+3\psi}.
\end{equation}
Moreover, for $q$ sufficiently large, we actually have the strict bound $A(q) \le q^{1-\delta}$ in that range.

We now decompose a general $q \in\mathscr{Q}(P_0)$ into factors of the correct size. Writing $A=\frac{n}{n-6\delta}+\varepsilon$ and $B=1+3\psi$, we find as in \cite{bde12} a decomposition of the form
\[q=q_1q_2\dots q_{t+1}\]
for each $q \in \mathscr{Q}(P_0)$ with $q_i$ pairwise coprime and so that $q_1,\dots,q_t$ are all in the range \eqref{eq.rangeforq} and $q_{t+1}<M^A$. Indeed, the only assumption needed for this iterative decomposition is that $2A \le B$ and $p_i^{k_i} \le M^B$, which is true by our assumptions.

It then follows that
\[A(q)=A(q_1)A(q_2)\dots A(q_t)A(q_{t+1}) \ll (q_1\dots q_t)^{1-\delta} q_{t+1} \ll M^A(q_1\dots q_t)^{1-\delta},\]
using our result from above and the trivial bound $\vert A(q_{t+1})\vert \le q_{t+1}$. With $q_0=q_1q_2\dots q_t$ it now follows that
\[R(P_0) \ll M^A \sum_{q_{t+1}<M^A} \sum_{q_0>\frac{P_0}{q_{t+1}}} q_0^{1-\delta} \ll M^A \sum_{q<M^A} \left(\frac{P_0}{q}\right)^{2-\delta} \ll M^{A\delta} P_0^{2-\delta}\]
as desired.
\end{proof}

We note that the condition $p^{k(p)} \le M^{1+3\psi}$ will be satisfied if $P_0 \le M^{1+3\psi}$ as well as $5n=70 \le 1+3\psi$, which we denote by $(\mathfrak{S}_2)$ and $(\mathfrak{S}_3)$.

\section{Synthesis}

\subsection{The major arc contribution}

Together with Lemma \ref{lem.majorarcsummary_vorsingulaerereihe} we can summarize the major arc contribution as follows:

\begin{lemma}
\label{lem.majorarcsummary_final}
In the case of a non-singular cubic form $C$ in $14$ variables, we have
\[\mathfrak{S}(P_0) \gg M^{-84-\varepsilon}\]
as soon as $P_0 \gg M^{259+\varepsilon}$.

In the case of an inhomogeneous cubic polynomial $\phi$ with $h \ge 14$, we have
\[\mathfrak{S}(P_0) \gg M^{-292(n^2-1)-\varepsilon}\]
as soon as $P_0 \gg M^{876(n^2-1)+7+\varepsilon}$.

Finally, in the case of a general cubic form in $14$ variables, we have
\[\mathfrak{S}(P_0) \gg M^{-84-\varepsilon}\]
if we assume $(\mathfrak{S}_2)$ and $(\mathfrak{S}_3)$ as well as
\begin{equation}
    \tag{$\mathfrak{S}_4$}
    P_0 \gg M^{\frac{1}{\delta-2} \cdot (84+\frac{14\delta}{14-6\delta})+\varepsilon}.
\end{equation}

Assuming additionally $(\mathfrak{M}_1)$, $(\mathfrak{M}_2)$ and $(\mathfrak{I}_1)$ as well as
\begin{equation}
    \tag{$\mathfrak{M}_3$}
    P_0^3u \ll \frac{P}{M^{8.5+T+\varepsilon}},
\end{equation}
we have in all three cases
\[\int_{\mathfrak{M}} S(\alpha) d\alpha \gg \frac{P^{n-3}}{M^{8.5+T+\varepsilon}}\]
where $T=84$ in the homogeneous case and $T=292(n^2-1)$ in the inhomogeneous case.
\end{lemma}

\subsection{The minor arc contribution}

We now use the different bounds obtained in Section \ref{sec.minorarcs} to bound the total minor arc contribution.

We dissect $\mathfrak{m}$ by an application of Dirichlet's Approximation Theorem for some parameter $Q$ to be determined. For every $\alpha \in \mathbb{R}$, this yields an approximation
\[\alpha=\frac{a}{q}+\theta \quad \text{with} \quad q \le Q, \vert \theta\vert \le \frac{1}{qQ}.\]
The assumption $\alpha \in \mathfrak{m}$ then implies that $q>P_0$ or $\theta > \frac{u}{P^3}$. Note that since the contribution to the minor arcs from the range $\vert \theta\vert \le \frac{1}{P^n}$ is clearly $\mathcal{O}(Q^2)$, we may also assume that $\vert \theta\vert \ge \frac{1}{P^n}$ for $q>P_0$. 

This allows us to apply a double dyadic decomposition with respect to both $\vert \theta\vert$ and $q$ which yields
\[\int_{\mathfrak{m}} S(\alpha) d\alpha \ll Q^2+P^{\varepsilon} \max_{R \le Q, \phi \le \frac{1}{RQ}} \Sigma(R,\phi),\]
where
\[\Sigma(R,\phi):=\sum_{q \sim R} \sum_{(a;q)=1} \int_{\vert \theta\vert \sim \phi} \left\vert S\left(\frac{a}{q}+\theta\right)\right\vert d\theta.\]
We note that the range of integration is a disjoint union of two intervals.

In view of the major arc contribution estimate from Lemma \ref{lem.majorarcsummary_final} it then suffices to show that 
\begin{equation}
    \label{eq.goalboundsigma}
    \Sigma(R,\phi) \ll \frac{P^{n-3}}{M^{8.5+T+\varepsilon}}
\end{equation}
if we add the harmless assumption
\begin{equation}
    \tag{$\mathfrak{m}_1$}
    Q^2 \ll \frac{P^{n-3}}{M^{8.5+T+\varepsilon}}.
\end{equation}

To employ the mean-value estimates developed in Section \ref{sec.minorarcs}, we apply the Cauchy-Schwarz inequality to obtain
\[\Sigma(R,\phi) \ll R\phi^{1/2} \left(\sum_{q \sim R} \sum_{(a;q)=1} \int_{\vert \theta\vert \sim \phi} \left\vert S\left(\frac{a}{q}+\theta\right)\right\vert^2 d \theta\right)^{1/2}.\]

We next cover the region $\vert \theta\vert \sim \phi$ by $\mathcal{O}\left(1+\frac{\phi}{\kappa}\right)$ intervals of size $\kappa$ centered at values $\alpha=\frac{a}{q}+\theta$ with $\vert \theta\vert \sim \phi$. We conclude that
\[\Sigma(R,\phi) \ll R\phi^{1/2} \left(1+\frac{\phi}{\kappa}\right)^{1/2} \left(\sum_{q \sim R} \sum_{(a;q)=1} M\left(\frac{a}{q}+\theta,\kappa\right)\right)^{1/2}\]
for some $\theta \sim \phi$.

Using \eqref{Malphakappa} and \eqref{eq.defA} we thus obtain
\[\Sigma(R,\phi) \ll R\phi^{1/2} \left(1+\frac{\phi}{\kappa}\right) \left(R^2+\frac{\kappa P^{\frac{3n}{2}-1+\varepsilon}}{H^{n-1}} A(\theta,R,H,P)\right)^{1/2}.\]
Using the bound \eqref{eq.boundA} for $A(\theta,R,H,P)$ we then find that
\begin{equation}
\label{eq.boundsigma}
    \Sigma(R,\phi) \ll R^2\phi^{1/2} \left(1+\frac{\phi}{\kappa}\right) \left(1+\kappa P^{2n-1+\varepsilon}HE\right)^{1/2}
\end{equation}
with
\[E=\frac{1}{H^{14}}+(RHM\eta)^7+\frac{\eta P^2}{(R\eta P^2)^7},\]
where we used the bound $\min(1,\eta P^2) \le \eta P^2$ which turns out to be sufficient. 

Suppose that we can show that $E \ll \frac{1}{H^{14}}$. Recalling that $\kappa \asymp \frac{(\log P)^2}{HP^2M^6}$, we have
\[1+\frac{\phi}{\kappa} \ll \frac{P^{\varepsilon}\eta}{\kappa}\]
from the definition \eqref{etadef} of $\eta$. As $\kappa \gg \frac{1}{P^n}$, we then obtain that both summands in the last bracket of \eqref{eq.boundsigma} are bounded by $\kappa P^{2n-1+\varepsilon}H^{-13}$. Still assuming $E \ll \frac{1}{H^{14}}$, we therefore find that
\[\Sigma(R,\phi) \ll R^2\phi^{1/2}\eta^{1/2} P^{n-\frac{1}{2}+\varepsilon} H^{-13/2}.\]
Recalling our goal \eqref{eq.goalboundsigma}, it thus suffices to have
\[H^{13} \gg M^{2T+17} R^4\phi^2P^{5+\varepsilon}\]
as well as
\[H^{14} \gg M^{2T+16} R^4\phi P^{3+\varepsilon}\]
in view of the definition of $\eta$.
We hence take
\[H \asymp P^{\varepsilon} \cdot \max\left( (M^{2T+17}R^4\phi^2P^5)^{1/13}, (M^{2T+16}R^4\phi P^3)^{1/14},1 \right).\]
We need to check whether this choice satisfies $H \le P$ and $H \le M^{\psi}$. The condition $H \le P$ will be satisfied if
\[M^{2T+17+\varepsilon}R^4\phi^2 \ll P^{8-\varepsilon} \quad \text{and} \quad M^{2T+16}R^4\phi \ll P^{11-\varepsilon}.\]
In view of $\phi R \le \frac{1}{Q}$ and $R \le Q$, it will therefore be satisfied if
\begin{equation}
    \tag{$\mathfrak{m}_2$}
    M^{2T+17} \ll P^{8-\varepsilon}
\end{equation}
and
\begin{equation}
    \tag{$\mathfrak{m}_3$}
    M^{2T+16} Q^2 \ll P^{11-\varepsilon}.
\end{equation}
Similarly, if $\psi<\infty$, the condition $H \le M^{\psi}$ will be satisfied when
\begin{equation}
    \tag{$\mathfrak{m}_4$}
    P^{5+\varepsilon} \ll M^{13\psi-2T-17}
\end{equation}
and
\begin{equation}
    \tag{$\mathfrak{m}_5$}
    P^{3+\varepsilon}Q^2 \ll M^{14\psi-2T-16}.
\end{equation}

Summarizing, we have found an admissible choice for $H$ that yields a satisfactory bound for the minor arc contribution under the assumption of $E \ll \frac{1}{H^{14}}$. We now need to enquire whether this condition is satisfied.

To this end, it is convenient to introduce the parameter
\[\phi_0:=(R^4P^{31}M^{2T+30})^{-1/15}.\]
One then readily checks that for $\phi \le \phi_0$, we have
\[H \asymp P^{\varepsilon} \max((M^{2T+16}R^4\phi P^3)^{1/14},1)\]
and
\[\eta \ll \frac{P^{\varepsilon}}{P^2HM}\]
while for $\phi \ge \phi_0$, we have
\[H \asymp P^{\varepsilon} \max((M^{2T+17}R^4\phi^2 P^5)^{1/13},1)\]
and
\[\eta \ll \phi.\]
To prove $E \ll \frac{1}{H^{14}}$ we need to check whether $RH^3M\eta \ll 1$ and $\left(\frac{H^2}{R\eta P^2}\right)^7 \eta P^2 \ll 1$. We begin with the first condition.

If $\phi \le \phi_0$, we have
\begin{align*}
    RH^3M\eta &\ll P^{\varepsilon} \frac{QH^2}{P^2}\\
    &\ll \frac{Q}{P^{2-\varepsilon}} \left(1+\left(M^{2T+16}R^4\phi P^3\right)^{1/7}\right)\\
    &\ll \frac{Q}{P^{2-\varepsilon}} +\left(\frac{M^{2T+16}Q^9}{P^{11-\varepsilon}}\right)^{1/7}.
\end{align*}
This is $\mathcal{O}(1)$ if we assume that 
\begin{equation}
    \tag{$\mathfrak{m}_6$}
    Q \ll \frac{P^{\frac{11}{9}-\varepsilon}}{M^{\frac{2T+16}{9}}}.
\end{equation}
If, conversely, $\phi \ge \phi_0$, we have
\begin{align*}
    RH^3M\eta &\ll RH^3M^{1+\varepsilon}\phi\\
    &\ll \frac{M^{1+\varepsilon}}{Q} \cdot \left(1+(M^{2T+17}R^4\phi^2 P^5)^{3/13}\right)\\
    &\ll \frac{M^{1+\varepsilon}}{Q}+\left(\frac{M^{6T+64}P^{15}}{Q^{13}}\right)^{1/13}
\end{align*}
which will be $\mathcal{O}(1)$ if
\begin{equation}
    \tag{$\mathfrak{m}_7$}
    Q \gg P^{\frac{15}{13}+\varepsilon}M^{\frac{6T+64}{13}}.
\end{equation}

We next turn to the question whether or not we have 
\begin{equation}
    \label{eq.lastcondition}
    \left(\frac{H^2}{R\eta P^2}\right)^7 \eta P^2 \ll 1.
\end{equation}

Again, let us first suppose that $\phi \le \phi_0$. Then $\eta \gg \frac{1}{P^2HM}$ so that
\[\left(\frac{H^2}{R\eta P^2}\right)^7 \eta P^2 \ll \left(\frac{H^3M}{R}\right)^7 \cdot \frac{1}{HM}= \frac{H^{20}M^6}{R^7}.\]

With our choice of $H$, this will be $\mathcal{O}(1)$ if $R^7 \gg M^6P^{\varepsilon}$ as well as $\phi \le \phi_1$, where
\[\phi_1:= \frac{R^{\frac{9}{10}}}{P^{3+\varepsilon}M^{2T+20.2}}.\]
If, conversely, $\phi \ge \phi_0$, we have $\eta \asymp \phi$ so that \eqref{eq.lastcondition} is equivalent to $H^{14} \ll R^7\phi^6P^{12}$.

With our choice of $H$ this will be satisfied as soon as $\phi \ge \phi_2$, where
\[\phi_2:= \frac{M^{\frac{7(2T+17)}{25}}}{R^{\frac{7}{10}}P^{\frac{43}{25}-\varepsilon}}.\]
To summarize, we have obtained a satisfactory bound for the minor arc contribution as soon as $R \gg M^{6/7}P^{\varepsilon}$ as well as $\phi \le \min(\phi_0,\phi_1)$ or $\phi \ge \max(\phi_0,\phi_2)$.

A quick computation shows that we will have $\phi_2 \le \phi_0 \le \phi_1$ as soon as $R \ge R_0$, where 
\[R_0:=P^{\frac{4}{5}+\varepsilon} M^{\frac{4}{5}(2T+17)+2}\]
and so in that case the assumptions are always satisfied, while for $R \le R_0$ we have $P^{-\varepsilon} \phi_1 \le \phi_0 \le P^{\varepsilon}\phi_2$. 

We are thus left to treat the case where $R \le R_0$ and $P^{-\varepsilon} \phi_1 \le \phi_0 \le P^{\varepsilon}\phi_2$ or $R \ll M^{6/7}P^{\varepsilon}$.

It is here that we use the bound obtained from the Weyl differencing. Noting that $(\mathfrak{m}_6)$ certainly implies $Q \le P^{3/2}$, applying Lemma \ref{lem.weylinequality} we find that
\[\Sigma(R,\phi) \ll R^2\phi P^{n+\varepsilon} \left(MR\phi+\frac{1}{R\phi P^3}+M^{-2\psi}\right)^{\frac{7}{4}}.\]
Recalling our goal \eqref{eq.goalboundsigma}, it will suffice to show that
\[R^2\phi P^3\left(MR\phi+\frac{1}{R\phi P^3}+M^{-2\psi}\right)^{\frac{7}{4}} \ll M^{-8.5-T-\varepsilon}.\]
This will be satisfied as soon as
\begin{equation}
    \label{eq.conditionphi_weyl}
    \frac{R^{\frac{1}{3}}M^{\frac{34+4T}{3}}}{P^{3-\varepsilon}} \ll \phi \ll \min\left\{\frac{1}{M^{\frac{41+4T}{11}}P^{\frac{12}{11}} R^{\frac{15}{11}}}, \frac{M^{\frac{7\psi}{2}-8.5-T+\varepsilon}}{R^2P^3}\right\}.
\end{equation}
Our bound is therefore satisfactory as soon as
\[\phi_1 \gg \frac{R^{\frac{1}{3}}M^{\frac{34+4T}{3}}}{P^{3-\varepsilon}} \]
as well as
\[\phi_2 \ll \min\left\{\frac{1}{M^{\frac{41+4T}{11}}P^{\frac{12}{11}} R^{\frac{15}{11}}}, \frac{M^{\frac{7\psi}{2}-8.5-T+\varepsilon}}{R^2P^3}\right\}.\]
The first condition will be satisfied as soon as $R \ge R_1$, where
\[R_1 \asymp M^{\frac{50}{17}(2T+17)+\frac{96}{17}},\]
whereas the second one will be satisfied if
\[R \ll \frac{P^{\frac{346}{365}-\varepsilon}}{M^{\frac{254(2T+17)+350}{365}}}\]
and
\[R \ll \frac{M^{\frac{35\psi}{13}-\frac{3}{5}(2T+17)}}{P^{\frac{64}{65}-\varepsilon}}.\]
The latter two conditions are satisfied for $R \le R_0$ if we assume
\begin{equation}
    \tag{$\mathfrak{m}_8$}
    P \gg M^{\frac{91}{9}(2T+17)+\frac{440}{27}+\varepsilon}
\end{equation}
and
\begin{equation}
    \tag{$\mathfrak{m}_9$}
    P \ll M^{\frac{175\psi}{116}-\frac{91(2T+17)}{116}-\frac{130}{116}-\varepsilon}.
\end{equation}

Finally, we are left to deal with the case where $R \le R_1$. Here, we need to use that we are on the minor arcs. Assuming
\begin{equation}
    \tag{$\mathfrak{m}_{10}$}
   P_0 \gg M^{\frac{50}{17}(2T+17)+\frac{96}{17}+\varepsilon},
\end{equation}
we may now assume that $R \le R_1 \le P_0$ and hence $\phi \ge \frac{u}{P^3}$.

It then suffices to check that again \eqref{eq.conditionphi_weyl} is satisfied. Using $\phi R \le \frac{1}{Q}$, this will be the case if
\begin{equation}
    \tag{$\mathfrak{m}_{11}$}
    Q \gg P^{\frac{12}{11}+\varepsilon} M^{\frac{234(2T+17)+503}{187}}
\end{equation}
and
\begin{equation}
    \tag{$\mathfrak{m}_{12}$}
    Q \gg \frac{P^3}{M^{\frac{7\psi}{2}-\frac{117(2T+17)+192}{17}}}
\end{equation}
as well as
\begin{equation}
    \tag{$\mathfrak{m}_{13}$}
    u \gg M^{\frac{28(2T+17)+32}{17}}.
\end{equation}

\subsection{The endgame}

We have now successfully bounded the minor arc contribution under all the assumptions $\mathfrak{m}_1$ up to $\mathfrak{m}_{13}$ and hence shown that $N(P)>0$, so that there is a solution $\vert \mathbf{x}\vert<P$.

Finally, we are ready to choose the parameters and deduce our theorems.

For Theorem \ref{thm.h=14asymptotic}, we may assume that $\psi=\infty$. We then take $M$ fixed and $u$ and $P_0$ to be a small power of $P$, we choose $Q=P^{\frac{7}{6}}$ and then let $P \to \infty$. It is then readily checked that all assumptions are indeed satisfied with this choice. Moreover, we have shown that the minor arcs contribute $\mathcal{O}(P^{n-3-\varepsilon})$ and the major arcs contribute the main term from the claimed asymptotic formula.

For Theorems \ref{thm.smallsolution_h=14} and the non-singular case of Theorem \ref{thm.smallsolution_hom}, we may still assume that $\psi=\infty$ so that all conditions involving $\psi$ are empty.

We choose $P_0=M^{\frac{50(2T+17)+96}{17}+\varepsilon}$ to satisfy $(\mathfrak{m}_{10})$. We also choose $u=M^{\frac{28(2T+17)+32}{17}+\varepsilon}$ to satisfy $(\mathfrak{m}_{13})$.

We then choose
\[P=M^{\frac{373(2T+17)+640}{34}+\varepsilon}\]
to satisfy $(\mathfrak{M}_{13})$. One now checks that all other conditions are also satisfied with this choice after choosing e.g. $Q=\frac{P^{11/9}}{M^{\frac{2T+16}{9}}}$.

We then plug in the values $T=292(n^2-1)$ and $T=84$, respectively, to deduce Theorem \ref{thm.smallsolution_h=14} and the non-singular case of Theorem \ref{thm.smallsolution_hom}.

Finally, we deduce Theorem \ref{thm.smallsolution_hom} in the case of a general cubic form. We choose $u$, $P_0$, $P$ and $Q$ as above where now $T=84$ so that our choice is
\[u \approx M^{306.58}, P_0 \approx M^{549.76}, P \approx M^{2048.38}.\]
Condition $(\mathfrak{m}_9)$ now requires $\psi \ge 1454.8$. 

We choose $\psi=1454.8$ and note that with $\delta=2.23$ all other conditions are now also satisfied.

If $C$ is $\psi$-good, we therefore have a solution $\vert \mathbf{x}\vert \le P \ll M^{2049}$. On the other hand, if $C$ is not $\psi$-good, by Lemma \ref{lemma_gc_hinvariante} we have a solution $\mathbf{x}\ll M^{97+91\psi} \ll M^{132484}$ as desired.

\section{List of assumptions}

\begin{equation}
\tag{$\mathfrak{M}_1$}
2P_0^2u<P^3
\end{equation}

\begin{equation}
\tag{$\mathfrak{M}_2$}
u \cdot P_0 \cdot M \cdot \vert \mathbf{z}\vert^2 \ll P
\end{equation}

\begin{equation}
    \tag{$\mathfrak{M}_3$}
    P_0^3u \ll \frac{P}{M^{8.5+T+\varepsilon}}
\end{equation}

\begin{equation}
    \tag{$\mathfrak{S}_1$}
    0<\frac{14}{14-6\delta}<1+3\psi
\end{equation}

\begin{equation}
    \tag{$\mathfrak{S}_2$}
    P_0 \le M^{1+3\psi}
\end{equation}

\begin{equation}
    \tag{$\mathfrak{S}_3$}
    \psi \ge 23
\end{equation}

\begin{equation}
    \tag{$\mathfrak{S}_4$}
    P_0 \gg M^{\frac{1}{\delta-2} \cdot (84+\frac{14\delta}{14-6\delta})+\varepsilon}
\end{equation}

\begin{equation}
    \tag{$\mathfrak{I}_1$}
    u^2M^{17+\varepsilon} \ll P
\end{equation}

\begin{equation}
    \tag{$\mathfrak{m}_1$}
    Q^2 \ll \frac{P^{n-3}}{M^{8.5+T+\varepsilon}}
\end{equation}

\begin{equation}
    \tag{$\mathfrak{m}_2$}
    M^{2T+17} \ll P^{8-\varepsilon}
\end{equation}
\begin{equation}
    \tag{$\mathfrak{m}_3$}
    M^{2T+16} Q^2 \ll P^{11-\varepsilon}
\end{equation}

\begin{equation}
    \tag{$\mathfrak{m}_4$}
    P^{5+\varepsilon} \ll M^{13\psi-2T-17}
\end{equation}
\begin{equation}
    \tag{$\mathfrak{m}_5$}
    P^{3+\varepsilon}Q^2 \ll M^{14\psi-2T-16}
\end{equation}
\begin{equation}
    \tag{$\mathfrak{m}_6$}
    Q \ll \frac{P^{\frac{11}{9}-\varepsilon}}{M^{\frac{2T+16}{9}}}
\end{equation}
\begin{equation}
    \tag{$\mathfrak{m}_7$}
    Q \gg P^{\frac{15}{13}+\varepsilon}M^{\frac{6T+64}{13}}
\end{equation}
\begin{equation}
    \tag{$\mathfrak{m}_8$}
    P \gg M^{\frac{91}{9}(2T+17)+\frac{440}{27}+\varepsilon}
\end{equation}
\begin{equation}
    \tag{$\mathfrak{m}_9$}
    P \ll M^{\frac{175\psi}{116}-\frac{91(2T+17)}{116}-\frac{130}{116}-\varepsilon}
\end{equation}
\begin{equation}
    \tag{$\mathfrak{m}_{10}$}
   P_0 \gg M^{\frac{50}{17}(2T+17)+\frac{96}{17}+\varepsilon}
\end{equation}
\begin{equation}
    \tag{$\mathfrak{m}_{11}$}
    Q \gg P^{\frac{12}{11}+\varepsilon} M^{\frac{234(2T+17)+503}{187}}
\end{equation}
\begin{equation}
    \tag{$\mathfrak{m}_{12}$}
    Q \gg \frac{P^3}{M^{\frac{7\psi}{2}-\frac{117(2T+17)+192}{17}}}
\end{equation}
\begin{equation}
    \tag{$\mathfrak{m}_{13}$}
    u \gg M^{\frac{28(2T+17)+32}{17}}
\end{equation}

\section{Acknowledgements}

This work was carried out while the author was a Ph.D. student at the University of Göttingen, supported by the DFG Research Training Group 2491 \lq Fourier Analysis and Spectral Theory\rq{}. I would like to thank my supervisor Jörg Brüdern for introducing me to the topic and for suggesting to look at the work of Davenport and Lewis. I am also grateful to Rainer Dietmann for providing me with a copy of the relevant part of Lloyd's thesis.

\printbibliography
\end{document}